\documentclass[12pt]{amsart}

\usepackage{mathscinet}
\usepackage[utf8]{inputenc} 
\usepackage[dvipsnames, table]{xcolor} 
\usepackage{color} 
\usepackage{tikz} 
\usepackage{tikz-cd} 
\usetikzlibrary{matrix,arrows,decorations.pathmorphing} 
\usepackage{amssymb, amsmath, mathtools, amsfonts, amsthm} 
\usepackage{stmaryrd}
\usepackage{thmtools,
  thm-restate} 

\usepackage{graphicx} 
\usepackage{subcaption} 
\usepackage{fullpage} 
\usepackage{caption} 
\usepackage{latexsym} 

\usepackage{enumitem} 
\usepackage[toc,page]{appendix} 
\usepackage{multicol} 
\usepackage{hyperref}
\definecolor{headercolor}{RGB}{255,255,240}
\definecolor{mylinkcolor}{RGB}{0,0,255}
\definecolor{mycitecolor}{RGB}{169,169,169}
\definecolor{myurlcolor}{RGB}{255,20,147}
\definecolor{santicolor}{rgb}{0.0, 0.8, 0.6}
\hypersetup{colorlinks=true,
urlcolor=myurlcolor,
citecolor=mycitecolor,
linkcolor=mylinkcolor,
breaklinks=true}
\usepackage{url} 
\usepackage{cleveref}

\usepackage[numbers]{natbib} 
\usepackage{setspace}
\usepackage{indentfirst} 
\usepackage{listings} 
\usepackage[mathscr]{eucal} 
\usepackage{manfnt} 
\usepackage{longtable} 
\usepackage{colonequals} 
\usepackage{nicefrac} 
\usepackage{wrapfig}
\usepackage{tkz-euclide} 
\usepackage{caption} 
\usepackage{subcaption} 
\usepackage{float}
\usepackage{soul} 
\usepackage{pgfplots}
\pgfplotsset{width=10cm,compat=1.9} 
\usepackage{listings}
\definecolor{codegreen}{rgb}{0,0.6,0}
\definecolor{codegray}{rgb}{0.5,0.5,0.5}
\definecolor{codepurple}{rgb}{0.58,0,0.82}
\definecolor{backcolour}{rgb}{0.95,0.95,0.92}

\lstdefinestyle{mystyle}{ backgroundcolor=\color{backcolour},   commentstyle=\color
 {codegreen}, keywordstyle=\color{magenta}, numberstyle=\tiny\color{codegray}, stringstyle=\color
 {codepurple}, basicstyle=\ttfamily\footnotesize, breakatwhitespace=false,         breaklines=true,
 captionpos=b,                    keepspaces=true,                 numbers=left, numbersep=5pt,
 showspaces=false,                showstringspaces=false, showtabs=false, tabsize=2 }

\newcounter{counter}[section] 

\numberwithin{equation}{section} 

\newtheorem{theorem}[counter]{Theorem}
\newtheorem{lemma}[counter]{Lemma}
\newtheorem{corollary}[counter]{Corollary}

\newtheorem{proposition}[counter]{Proposition}

\theoremstyle{definition} \newtheorem{definition}[counter]{Definition}
\newtheorem{example}[counter]{Example} 

\theoremstyle{remark} 

\newtheorem{remark}[counter]{Remark}

\newcommand{\R}{\mathbb{R}} 
\newcommand{\Z}{\mathbb{Z}} 
\newcommand{\Q}{\mathbb{Q}} 
\newcommand{\C}{\mathbb{C}} 
\renewcommand{\P}{\mathbb{P}} 
\newcommand{\Qbar}{\bar{\Q}} 
\newcommand{\XX}{\mathscr{X}} 
\newcommand{\YY}{\mathscr{Y}} 
\newcommand{\ZZ}{\mathscr{Z}} 
\newcommand{\FF}{\mathscr{F}} 
\newcommand{\GG}{\mathscr{G}} 
\newcommand{\PP}{\mathscr{P}} 
\newcommand{\EE}{\mathcal{E}} 

\newcommand{\Pone}{\P^1} 
\newcommand{\md}{\text{ mod }} 
\newcommand{\unit}{^{\times}} 
\newcommand{\dd}{\,\mathrm{d}} 
\newcommand{\paren}[1]{\left( #1 \right)} 
\newcommand{\brk}[1]{\left\lbrace #1 \right\rbrace} 
\newcommand{\ideal}[1]{\langle #1 \rangle} 
\newcommand{\cdef}[1]{{\color{blue}\textsf{#1}}} 
\renewcommand{\Re}{\mathrm{Re}} 
\renewcommand{\Im}{\mathrm{Im}} 
\newcommand\doublelong[2]{\mathbin{\xymatrix{{}\ar@<3pt>[r]^{#1}
\ar@<-3pt>[r]_{#2}&}}} 
\newcommand{\thickslash}{\mathbin{\!\!\pmb{\fatslash}}} 
\newcommand{\rig}{\thickslash} 
\newcommand{\oalpha}{\bar{\alpha}}
\newcommand{\mfa}{\mathfrak{a}}
\newcommand{\mfp}{\mathfrak{p}}
\newcommand{\mfpb}{\bar{\mathfrak{p}}}
\newcommand{\mfq}{\mathfrak{q}}
\newcommand{\II}{\mathcal{I}}
\newcommand{\DD}{\mathbb{D}}
\newcommand{\RR}{\mathcal{R}}
\newcommand{\x}{\mathsf{x}}
\newcommand{\y}{\mathsf{y}}
\newcommand{\z}{\mathsf{z}}
\newcommand{\TMD}{\mathsf{TMD}}

\newcommand{\fft}{(4,4,2)}

\newcommand{\mand}{\text{ and }} 
\newcommand{\mfor}{\text{ for }} 
\newcommand{\mif}{\text{ if }} 

\newcommand{\Gm}{\mathbb{G}_m} 

\let\originalleft\left \let\originalright\right
\renewcommand{\left}{\mathopen{}\mathclose\bgroup\originalleft}
  \renewcommand{\right}{\aftergroup\egroup\originalright}

\definecolor{headercolor}{RGB}{255,255,240}
\definecolor{yes}{rgb}{0.31, 0.78, 0.47}
\definecolor{no}{rgb}{1.0, 0.41, 0.38}

\definecolor{orchid}{rgb}{0.85, 0.44, 0.84}

\DeclareMathOperator{\Nm}{Nm} 
\DeclareMathOperator{\ord}{ord} 
\DeclareMathOperator{\Spec}{Spec} 
\DeclareMathOperator{\Sym}{Sym} 
\DeclareMathOperator{\Proj}{Proj} 
\DeclareMathOperator{\vol}{vol} 
\DeclareMathOperator{\Res}{Res} 
\DeclareMathOperator{\sheafHom}{\mathscr{H}\text{\kern -1pt
    {om}}} 
\DeclareMathOperator{\sProj}{\mathscr{P}\text{\kern -1pt
    {roj}}} 
\DeclareMathOperator{\mind}{md} 
\DeclareMathOperator{\twmind}{tmd} 
\DeclareMathOperator{\tw}{tw}
\DeclareMathOperator{\Ht}{Ht} 
\DeclareMathOperator{\twHt}{twHt} 
\DeclareMathOperator{\PProj}{\mathbf{Proj}}

\usepackage[textsize=tiny]{todonotes}
\usepackage{xargs, xpatch}
\title{Counting 5-isogenies of elliptic curves over $\mathbb{Q}$}

\author{Santiago Arango-Piñeros} \address{Department of Mathematics,
  Emory University, Atlanta, GA 30322, USA}
\email{santiago.arango.pineros@gmail.com}
\urladdr{\url{https://sarangop1728.github.io/}}

\author{Changho Han}
\address{Department of Mathematics, Korea University, Seoul, South Korea}
\email{changho\_han@korea.ac.kr}
\urladdr{\url{https://sites.google.com/view/changho-han}}

\author{Oana Padurariu} \address{Max-Planck-Institut f\"ur Mathematik, Bonn, Germany}
\email{opadurariu@mpim-bonn.mpg.de}
\urladdr{\url{https://sites.google.com/view/oanapadurariu/home}}

\author{Sun Woo Park} \address{Max-Planck-Institut f\"ur Mathematik, Bonn, Germany}
\email{s.park@mpim-bonn.mpg.de}
\urladdr{\url{https://sites.google.com/wisc.edu/spark483}}

\thanks{This work originated with discussions at a meeting of the AMS
  Mathematics Research Community \emph{Explicit Computations with Stacks} supported by
  the NSF under Grant Number DMS 1916439.}

\begin{document}

\begin{abstract}
  We show that the number of $5$-isogenies of elliptic curves defined over $\mathbb{Q}$ with naive height bounded by $H > 0$ is asymptotic to $C_5\cdot H^{1/6} (\log H)^2$ for some explicitly computable constant $C_5 > 0$. This settles the asymptotic count of rational points on the genus zero modular curves $\mathscr{X}_0(m)$ (as moduli stacks). We leverage an explicit $\mathbb{Q}$-isomorphism between the stack $\mathscr{X}_0(5)$ and the generalized Fermat equation $x^2 + y^2 = z^4$ with $\mathbb{G}_m$-action of weights $(4, 4, 2)$.
\end{abstract}

\setcounter{tocdepth}{1}
\maketitle
\tableofcontents



\section{Introduction}
\subsection{Setup: arithmetic statistics of elliptic curves}
Let $E$ be an elliptic curve defined over $\Q$. Then $E$ is isomorphic to a
unique elliptic curve $E_{A,B}$ with Weierstrass equation of the form
\begin{equation*}
    y^2 = x^3 + Ax + B,
\end{equation*}
where $A, B$ are integers, the discriminant
$\Delta(A,B) \colonequals -16(4A^3 + 27B^2)$ is nonzero, and no prime $\ell$
satisfies $\ell^4 \mid A$ and $\ell^6 \mid B$.

Let $\mathcal{E}$ denote the set of such elliptic curves. The \cdef{naive height} of
$E \cong E_{A,B} \in \mathcal{E}$ is defined by
\begin{equation} \label{eq:ec-height}
    \Ht(E) \colonequals \Ht(E_{A,B}) \colonequals \max(4|A|^3, 27|B|^2).
\end{equation}
For $H \geq 1$, define 
\begin{equation}
  \label{eq:EE-leg-H}
  \mathcal{E}_{\leq H} \colonequals \{E \in \mathcal{E} : \Ht(E) \leq H
  \}.
\end{equation}

Recent work has resolved many instances of counting problems for elliptic
curves equipped with additional level structure as $H \to \infty$. For
instance, Harron–Snowden \cite{Harron&Snowden17} and Cullinan–Kenney–Voight
\cite{Cullinan&Kenney&Voight22}, building on work by Duke \cite{Duke97} and
Grant \cite{Grant00}, provided asymptotics for the count of
$E \in \mathcal{E}_{\leq H}$ where the torsion subgroup
$E(\mathbb{Q})_{\mathrm{tors}}$ of the Mordell–Weil group is isomorphic to a
given finite abelian group, for each of the 15 groups described in
Mazur’s torsion theorem \cite[Theorem 2]{Mazur78}. These cases correspond to the modular curves $\YY_1(m)$ with coarse genus zero.

Parallel to these cases, a class of modular curves that has received much
attention is the modular curves $\mathscr{Y}_0(m)$. 
These curves are moduli stacks classifying \cdef{rational cyclic
  $m$-isogenies}, and their coarse moduli schemes are the classical modular curves $Y_0(m)$. Explicitly, every $\Q$-point of $\mathscr{Y}_0(m)$ can be thought of as a pair $(E,C)$ where $E$ is an elliptic curve defined over $\Q$, and
$C \subset E(\Qbar)$ is a cyclic subgroup of order $m$ that is stable under the
action of the absolute Galois group. Two such pairs $(E,C)$ and $(E',C')$ are
said to be $\Q$-isomorphic, if there exists an isomorphism
$\varphi\colon E \to E'$ of elliptic curves, defined over $\Q$, such that
$\varphi(C) = C'$. Let $\EE_m \colonequals \YY_0(m)\ideal{\Q}$ denote the set of isomorphism
classes $[E,C]$ of rational cyclic $m$-isogenies, and consider the finite 
subsets
\begin{equation*}
  \EE_{m, \leq H} \colonequals \brk{[E,C] \in \EE_m : \Ht(E) \leq H}.
\end{equation*}

As a corollary of Mazur's \emph{Isogeny Theorem} \cite[Theorem 1]{Mazur78}, it
is known that $Y_0(m)$ has infinitely many rational points if and only if
\begin{equation}
  \label{eq:m-genus-zero-X0(m)}
  m
\in\brk{1,2,3,4,5,6,7,8,9,10,12,13,16,18,25}.
\end{equation}
These are precisely the levels $m$ for which the compactified curve $X_0(m)$ has
genus zero.

\subsection{Results}
\label{sec:resutls} By recent work of several authors (see \Cref{tab:N_m(H)}),
the asymptotic order of growth of the counting function
$N_m(H) \colonequals \# \EE_{m,\leq H}$ was computed for every $m$ in this list,
except for the stubborn level $m = 5$, which had eluded all previous attempts.
When $m=5$, the previous best known estimate was $N_5(H) \asymp H^{1/6}(\log H)^2$ by work
of Boggess--Sankar \cite[Proposition 5.14]{Boggess&Sankar24}, but the explicit value of 
the leading coefficient was left unknown.
Our main theorem provides the asymptotic count of $5$-isogenies of elliptic
curves, including explicit value of leading coefficient.

\begin{theorem}
  \label{thm:main-result}
  There exist explicitly computable constants $C_5, C_5', C_5''\in \R$ and
  $c_5 \in (0,1)$ such that
  \begin{equation*}
    N_5(H) = C_5 H^{1/6}(\log H)^2 + C_5' H^{1/6}(\log H) +
    C_5'' H^{1/6} + O(H^{1/6} \cdot (\log H)^{-c_5}),
  \end{equation*}
  as $H\to \infty$. Furthermore, the constant $C_5$ is given by
  \begin{equation}
    \label{eq:C5}
    C_5 = \frac{41}{1536 \pi} \cdot V_5 \cdot g_1,
  \end{equation}
  where $V_5$ is given by \Cref{eq:V_5}, and $g_1$ is given by \Cref{eq:g1}.
\end{theorem}

\begin{remark}
    The nature of our methods prevent us from obtaining a power-saving error term for $N_5(H)$. We believe the error term can be improved to $O_\varepsilon(H^{1/12 + \varepsilon})$.
\end{remark}

In an upcoming paper by Alessandro Languasco and Pieter Moree \cite{Languasco&Moree25}, and independently by Steven Charlton, have numerically estimated the
constants $V_5$ and $g_1$ to extremely high precision, from which we can obtain

\begin{equation}
    C_5 \approx 0.00034377729776664529\dots 
\end{equation}

Our strategy was
inspired by their idea of exploiting explicit presentations for the ring of
modular forms of $\Gamma_0(5)$. In fact, we refine their estimate and show that
\Cref{thm:main-result} is essentially equivalent to the following.
\begin{theorem}[{\Cref{thm:GG-count}}]
    \label{thm:GG-count-main}
    Let $N_\GG(H)$ denote the counting function of the integer triples
    $(a,b,c)$ satisfying the following properties:
    \begin{itemize}
    \item $a^2 + b^2 = c^4$,
    \item $\gcd(a,b)$ is $4^{\text{th}}$-power free, and
    \item $|c| \leq H$.
    \end{itemize}
    Then, there exist explicitly computable constants $g_1, g_2, g_3 \in \R$
    such that for every $\varepsilon > 0$, we have
    $$N_{\GG}(H) = g_1 H(\log H)^2 + g_2H(\log H) + g_3H + O_\varepsilon(H^{1/2+\varepsilon}),$$
    as $H \to \infty$. Furthermore, the constant $g_1$ is given by the Euler product
    \begin{equation}
      \label{eq:g1}
      g_1  = \frac{48}{\pi^4}\prod_{p\, \equiv\, 1 \md 4} 
      \Bigl(1 - \frac{4}{(p+1)^2} \Bigr).
      \end{equation}
    \end{theorem}
    Languasco and Moree \cite{Languasco&Moree25} have numerically estimated the
    constant $g_1$ to extremely high precision, obtaining
    \begin{equation}
      \label{eq:g1-approx}
           g_1  \approx 0.41662592496198471660\dotsc.
    \end{equation}

Adding our main result to the list, we now know the leading terms in the 
asymptotic counts of rational points for the genus zero modular curves
$\mathscr{X}_0(m)$, which are compactifications, as moduli stacks, of $\mathscr{Y}_0(m)$ such that
the coarse moduli spaces are $X_0(m)$.

\begin{theorem}
  Suppose that $m$ is in the list (\ref{eq:m-genus-zero-X0(m)}). Then, there
  exist explicitly computable constants $C_m \in \R_{>0}$, $a_m \in \Q_{\geq 0}$, and
  $b_m \in \Z_{\geq 0}$ such that
  \begin{equation*}
    N_m(H) \sim C_m H^{a_m}(\log H)^{b_m}.
  \end{equation*}
  Moreover, the values of $a_m$ and $b_m$ are presented in \Cref{tab:N_m(H)}.
\end{theorem}

\begin{table}[ht]
  \caption{Powers of $H$ and $\log H$ in the main terms in the asymptotic of the
  counting functions $N_m(H)$.} 
  \setlength{\arrayrulewidth}{0.2mm}
  \setlength{\tabcolsep}{5pt}
  \renewcommand{\arraystretch}{1.2} \centering
  \begin{tabular}{|c|c|c|l|}
    \hline
    \rowcolor{headercolor}
    $m$ & $a_m$ & $b_m$ & Reference  \\ \hline
    1 & 5/6 & 0 & {\cite[Lemma 4.3]{Brumer92}} \\ \hline
    2 & 1/2 & 0 & {\cite[Proposition 1]{Grant00}} \\ \hline
    3 & 1/2 & 0 & {\cite[Theorem 1.3]{Pizzo&Pomerance&Voight20}} \\ \hline
    4 & 1/3 & 0 & {\cite[Theorem 4.2]{Pomerance&Schaefer21}} \\ \hline
    5 & 1/6 & 2 & \Cref{thm:main-result} \\ \hline
    6 & 1/6 & 1 & {\cite[Theorem 1.2.2]{Phillips22} \cite[Theorem 7.3.14]{Molnar24}} \\ \hline
    7 & 1/6 & 1 & {\cite[Theorem 1.2.2]{Molnar&Voight23}} \\ \hline
    8 & 1/6 & 1 & {\cite[Theorem 1.2.2]{Phillips22} \cite[Theorem 7.3.14]{Molnar24}} \\ \hline
    9 & 1/6 & 1 & {\cite[Theorem 1.2.2]{Phillips22} \cite[Theorem 7.3.14]{Molnar24}} \\ \hline
    10 & 1/6 & 0 & {\cite[Theorem 5.4.11]{Molnar24}} \\ \hline
    12 & 1/6 & 1 & {\cite[Theorem 1.2.2]{Phillips22} \cite[Theorem 7.3.18]{Molnar24}} \\ \hline
    13 & 1/6 & 0 & {\cite[Theorem 6.4.5]{Molnar24}} \\ \hline
    16 & 1/6 & 1 & {\cite[Theorem 7.3.18]{Molnar24}} \\ \hline
    18 & 1/6 & 1 & {\cite[Theorem 7.3.18]{Molnar24}} \\ \hline
    25 & 1/6 & 0 & {\cite[Theorem 5.4.11]{Molnar24}} \\ \hline
\end{tabular}
\label{tab:N_m(H)} 
\end{table}

\begin{remark}
  Since some elliptic curves admit multiple non isomorphic cyclic
  $m$-isogenies, the problem of counting elliptic curves that \emph{admit} an
  isogeny is related but distinct. In particular, \cite{Molnar24} denotes our
  counting function by $\widetilde{N}_m(H)$. For $m=5$, a single elliptic curve
  $E$ can admit at most two distinct $\Q$-rational $5$-isogenies. See
  \cite{Chiloyan&Lozano-Robledo21}.
\end{remark}

\begin{remark}
    For both torsion and isogenies, the work of Bruin--Najman \cite{BruinNajman22} and Phillips \cite{Phillips22} extend these asymptotic counts to the case of number fields when the corresponding moduli stack is a weighted projective line.
\end{remark}

For $m \neq 5$, the strategy was roughly the following. Use a rational
parametrization $\Pone \cong X_0(m)$ to find polynomial equations $A_m(t)$ and
$B_m(t)$ for the Weierstrass coefficients of elliptic curves admitting a
rational cyclic $m$-isogeny. Homogenize these polynomials to turn the
estimation of $N_m(H)$ into a lattice point counting problem in the compact
region
\begin{equation*}
  R_{m,\leq H} \colonequals \brk{(x,y) \in \R^2
    :\max(4|A_m(x,y)|^3,27B_m(x,y)^2)\leq H},
\end{equation*}
and use the principle of Lipschitz and a careful sieve to conclude the results.
This approach is sufficient when the expected power of $\log H$ in the main
term of $N_m(H)$ is $b_m = 0$. This is not surprising, since the main term
predicted by the principle of Lipschitz is the area of the region
$R_{m,\leq H}$, which is $\asymp H^a$.

Notably, Molnar--Voight \cite{Molnar&Voight23} where able to salvage this
approach to show that
\begin{equation*}
  N_7(H) \sim C_7 H^{1/6}(\log H).
\end{equation*}
They observed that counting $7$-isogenies up to $\Qbar$-isomorphism has the
effect of removing the $\log H$ factor, which allowed them to use Lipschitz for
this \emph{rigidified} count, and then recover the full count by precisely estimating
the number of twists with a given height (see \cite[Theorem 5.1.4]{Molnar&Voight23}).

Unfortunately, this rigidification method is not enough to count $5$-isogenies
essentially because we can only remove one factor of $\log H$. Nevertheless,
we use their rigidification method to simplify the calculations. For a detailed
discussion of why this method fails for $m=5$, see Remark 1.3.2 and Remark
5.3.53 in Molnar's Ph.D. thesis \cite{Molnar24}.

\subsection{Sketch of the proof}
\label{sec:proof-sketch}

Our approach does not use a rational parametrization of $X_0(5)$. Instead, we
leverage an explicit embedding of the moduli stack $\XX_0(5)$ in weighted
projective stack. This allows us to parametrize $5$-isogenies in terms of
\cdef{\fft-minimal} integral solutions to the generalized Fermat equation
$x^2+y^2=z^4$. (An integer triple $(a,b,c)$ is \fft-minimal, if
there is no prime $\ell$ satisfying $\ell^4 \mid a$, $\ell^4 \mid b$, and
$\ell^2 \mid c$).

\begin{remark}[Ignoring stacks]
   Keeping stacks in the back of one's mind is enlightening but not crucial for
   reading this article. If the reader wishes to understand our proof while
   ignoring stacks completely, they can safely skip \Cref{sec:X0(5)}.
 \end{remark}

\begin{remark}[Embracing stacks]
  The stacky perspective offers a conceptual framework ideally suited to
  studying counting problems of this nature. The ideas presented in
  \Cref{sec:X0(5)} are applicable in other contexts and might provide some
  insight into the Batyrev--Manin--Malle conjectures for rational points on
  stacks as in \cite{Ellenberg&Satriano&ZB23},\cite{DardaYasuda23}, \cite{Darda&Yasuda22}, \cite{Loughran&Santens24}.
\end{remark}
 
We break the proof into four steps and dedicate one section to
each one. 

\begin{enumerate}
\item In \Cref{sec:X0(5)}, we study the geometry of the moduli stack $\XX_0(5)$
  (defined over $\Q$) with coarse space $X_0(5)$. We show that $\XX_0(5)$ is
  $\Q$-isomorphic to the closed substack of the weighted projective stack
  $\PP(4,4,2)$ defined by equation $x^2 + y^2 = z^4$. The key result of this section
  towards the proof of our main theorem is \Cref{lemma:XX05=GG}.
\item In \Cref{sec:counting-ideals}, we count $(4,4,2)$-minimal Gaussian integers
  with bounded norm. We use a zeta function approach and the \emph{factorization method}, as explained in \cite{Alberts24}.
\item In \Cref{sec:counting-gaussian-ideals}, we prove the \emph{Main analytic lemma}
  (\Cref{lemma:main-analytic-lemma}). This enables us to count integral ideals
  in the Gaussian integers with numerical norm inside a homogeneously expanding
  region $\Omega$, as long as $\Omega$ is similar enough to a square. This section resides in the realm of complex analysis, and is independent from the rest of the article.
\item In \Cref{sec:proof} we feed our parametrization of $5$-isogenies and the
  count of $(4,4,2)$-minimal Gaussian integers of bounded norm into the \emph{Main analytic
    lemma}, completing the proof of \Cref{thm:main-result}.
\end{enumerate}

\subsection*{Acknowledgments}
We sincerely thank the organizers of 2023 AMS Mathematics Research Communities, Explicit Computations with Stacks, for providing a wonderful environment from which this collaborative project started. This project is generously supported by the National Science
Foundation Grant Number 1916439 for the 2023 Summer AMS Mathematics Research Communities.
We thank John Voight for getting this project started and for sharing his code to compute canonical rings. We thank David Zureick-Brown, and Jordan Ellenberg for helpful
conversations about this project. We thank Robert Lemke Oliver for suggesting
the approach presented in \Cref{sec:counting-gaussian-ideals}. We thank
Alessandro Languasco and Pieter Moree for their insightful comments on an
earlier draft, and for sharing their computations with us. We thank Steven Charlton for sharing his computations with us.
C.H. is very grateful to Korea University Grants K2422881, K2424631, K2510471, and K2511121 for the partial financial support.
O.P. and S.W.P are
very grateful to the Max-Planck-Institut f\"ur Mathematik Bonn for their
hospitality and financial support.

\section{Embedding \texorpdfstring{$\XX_0(5)$}{XX0(5)} into \texorpdfstring{$\PP(4,4,2)$}{P(4,4,2)}}
\label{sec:X0(5)}

We work in the level of generality required for our applications. In
particular, the geometric objects in this section are defined over $\Q$.
\subsection{The stacky proj construction}
\label{sec:stacky-proj}
We recall the stacky proj construction \cite[Example 10.2.8]{Olsson16}, specializing to the case where the
base scheme is $\Spec \Q$. Given a graded $\Q$-algebra
$R = \bigoplus_{d \geq 0} R_d$, the multiplicative group $\Gm$ acts on
$\Spec R$. The action is determined by the grading, and it fixes the point
$0 \colonequals V(R_+)$ corresponding to the irrelevant ideal
$R_+ \colonequals \bigoplus_{d > 0} R_d$. Define the \cdef{stacky proj of $R$} to be
the quotient stack
$$\PProj R \colonequals [(\Spec R - 0)/\Gm].$$ 

The stack $\PProj R$ admits a coarse space, namely the scheme $\Proj R$.
Moreover, if $R$ is generated in degree one, then $\PProj R = \Proj R$.

\begin{example}[The moduli stack of elliptic curves]
  \label{ex:XX(1)}
  Consider the graded algebra $\Q[u^4\x,u^6\y]$, where the grading is
  determined by the dummy variable $u$. By definition, the weighted projective
  line $\PP(4,6)$ is $\PProj \Q[u^4\x,u^6\y]$. Using Weierstrass equations, one
  can show that $\PP(4,6)$ is isomorphic to the moduli stack $\XX(1)$ of stable
  elliptic curves over $\Q$. In particular, the groupoid of rational points
  $\XX(1)(\Q)$ is equivalent to the groupoid $\PP(4,6)(\Q)$ with
  \begin{itemize}
  \item \textbf{Objects:} pairs $(A,B) \in \Q\times \Q - (0,0)$.
  \item \textbf{Isomorphisms:} elements $\lambda \in \Q^\times$, sending $(A,B) \mapsto (\lambda^6A, \lambda^4 B)$.
  \end{itemize}
\end{example}

\begin{example}
  \label{ex:GG}
  Consider the graded algebra
  $R \colonequals \Q[u^4\x, u^4\y, u^2\z]/(u^8(\x^2 + \y^2 - \z^4))$, and let
  $\GG \colonequals \PProj R$. The quotient map $\Q[u^4\x, u^4\y, u^2\z] \to R$
  induces a closed embedding $\GG \hookrightarrow \PP(4,4,2)$. We will show that
  $\GG$ is isomorphic to the moduli space $\XX_0(5)$ of generalized elliptic
  curves together with $5$-isogenies. In particular, the groupoid of rational
  points $\XX_0(5)(\Q)$ is equivalent to the groupoid $\GG(\Q)$ with
    \begin{itemize}
    \item \textbf{Objects:} triples
      $(a,b,c) \in \Q\times \Q \times \Q - (0,0,0)$ satisfying $a^2 + b^2 = c^4$.
    \item \textbf{Isomorphisms:} elements $\lambda \in \Q^\times$, sending
      $(a,b,c) \mapsto (\lambda^4a, \lambda^4 b, \lambda^2 c)$.
    \end{itemize}
\end{example}

We now explain special properties of $\Proj R$ that are mentioned in \cite[Section 2]{Abramovich&Hasset11}. When
the graded $\Q$-algebra $R$ satisfies $R_0 = \Q$, the stack $\PProj R$ is
special, as the stabilizers of any point of $\PProj R$ is a finite subgroup of
$\Gm$. If in addition $R$ is a finitely generated graded $\Q$-algebra, then
$\PProj R$ is an example of a \cdef{cyclotomic stack} (see \cite[Definition 2.3.1]{Abramovich&Hasset11}). Just
like $\PProj R$, any cyclotomic stack $\XX$ has a coarse moduli space $X$. Note
that both $\XX(1)$ and $\GG$ from \Cref{ex:XX(1)} and \Cref{ex:GG} are
cyclotomic stacks.

Recall that if the graded $\Q$-algebra $R$ is finitely generated by elements of
$R_1$ and $R_0 = \Q$, then there is a closed immersion
$\Proj R \hookrightarrow \P R_1 \colonequals \Proj \Sym R_1$. In this case, the line bundle
$\mathcal{O}_{\Proj R}(1) = \mathcal{O}_{\P R_1}(1)|_{\Proj R}$ can be
recovered from the graded $R$-module $R(1)$. Similarly, if $R_0 = \Q$ but $R$
is not necessarily generated by $R_1$, then $\PProj R$ is equipped with a line
bundle $\mathcal{O}_{\PProj R}(1)$; the pullback of $\mathcal{O}_{\PProj R}(1)$
via $\Spec R - 0 \to \PProj R$ is the $\widetilde{R(1)}|_{\Spec R - 0}$, where
$\widetilde{R(1)}$ is the sheaf associated to the graded $R$-module $R(1)$.
Note that $\Gm$ acts faithfully on the line bundle
$\widetilde{R(1)}|_{\Spec R - 0}$ over the $\Gm$-variety $\Spec R - 0$. Consequently, the
stabilizer groups of any point of $\PProj R$ act faithfully on the
corresponding fiber of $\mathcal{O}_{\PProj R}(1)$; such a line bundle is
called a \cdef{uniformizing line bundle} (see \cite[Definition 2.3.11]{Abramovich&Hasset11}).

Just as $\mathcal{O}_{\Proj R}(1)$ is an ample line bundle on $\Proj R$,
$\mathcal{O}_{\PProj R}(1)$ is more than just a uniformizing line bundle.
\begin{definition}[{\cite[Definition 2.4.1]{Abramovich&Hasset11}}] 
\label{def:polarizing_LB}
Suppose that $\XX$ is a proper cyclotomic $\Q$-stack with coarse moduli space
$X$. Denote by $c \colon \XX \to X$ the coarse map. Then a uniformizing line bundle
$\mathcal{L}$ on $\XX$ is called a \cdef{polarizing line bundle} if there
exists an ample line bundle $M$ on $X$ and a positive integer $e$ such that
$\mathcal{L}^e \cong c^*M$.
\end{definition}

We claim that if $R$ is a finitely generated graded $\Q$-algebra such that
$R_0 = \Q$ and $R_{1,m} \colonequals \oplus_{1 \le d \le m} R_d$ generates $R$, then
the uniformizing line bundle $\mathcal{O}_{\PProj R}(1)$ is a polarizing line
bundle; this claim immediately implies that $\mathcal{O}_{\XX(1)}(1)$ and
$\mathcal{O}_\GG(1)$ are polarizing line bundles on $\XX(1)$ and $\GG$,
respectively. To see this, recall that $\Proj R$ is a closed subscherme of a
weighted projective space $\P R_{1,m} \colonequals \Proj R'$, where $R'$ is a
free graded $\Q$-algebra generated by the graded $\Q$-vector space $R_{1,m}$;
note that $R$ is a graded quotient of $R'$. Then, there exists a positive
integer $e$ such that $\mathcal{O}_{\P R_{1,m}}(e)|_{\Proj R}$ is a line bundle on $\Proj R$, which implies that $\mathcal{O}_{\PProj R}(e) \cong c^*(\mathcal{O}_{\P R_{1,m}}(e)|_{\Proj R})$;
here, $e$ is not necessarily equal to $1$ as $\mathcal{O}_{\P R_{1,m}}(k)$ may
not be locally free for some values of $k$.

\subsection{Rigidification}
\label{sec:rigidification} Recall the notion of rigidification, as presented for instance in \cite[Appendix C]{Abramovich&Graber&Vistoli08}. When
$\XX = \PProj R$, the rigidification construction is very explicit. Indeed, let
$d$ be the greatest common divisor of the degrees of the generators of $R$, so
that $R = \bigoplus_{n\geq 0} R_{nd}$. Note that
$\mu_d \colonequals \Spec \Q[t]/(t^d-1) \subset \Gm$ acts trivially on
$\Spec R$. Consider the graded ring $S \colonequals \bigoplus_{n\geq 0} S_{n}$ defined by
$S_n \colonequals R_{nd}$. We have a homomorphism $S \to R$ which is the identity
at the level of rings, but multiplies the degree of every homogeneous element
of $S$ by $d$. Then, the corresponding morphism of stacks
$\PProj R \to \PProj S$ is a $\mu_d$-gerbe, and $\XX \rig \mu_d = \PProj S$ is the \emph{rigidification}
of $\XX$ by $\mu_d$-inertia. In this case, the pullback of
$\mathcal{O}_{\PProj S}(1)$ under the rigidification $\PProj R \to \PProj S$ is
$\mathcal{O}_{\PProj R}(d)$.

\begin{example}[The rigidified moduli space of elliptic curves]
  \label{ex:ZZ(1)}
  The rigidification $\ZZ(1) \colonequals \XX(1) \rig \mu_2$ is isomorphic to
  $\PP(2,3)$. In particular, the groupoid of rational points $\ZZ(1)(\Q)$ is
  equivalent to the groupoid $\PP(2,3)(\Q)$ with
  \begin{itemize}
  \item \textbf{Objects:} pairs $(A,B) \in \Q\times \Q - (0,0)$.
  \item \textbf{Isomorphisms:} elements $\lambda \in \Q^\times$, sending
    $(A,B) \mapsto (\lambda^2A, \lambda^3 B)$.
  \end{itemize}
\end{example}

\begin{example}
  \label{ex:FF}
  The rigidification of the graded $\Q$-algebra $R$ from \Cref{ex:GG} is
  $S \colonequals \Q[u^2\x, u^2\y, u\z]/(u^4(\x^2 + \y^2 - \z^4))$. Let
  $\FF \colonequals \PProj S$. Then $\GG \to \FF$ is a $\mu_2$-gerbe, and $\FF$ is
  the rigidification of $\GG$. We will show that $\FF$ is isomorphic to
  $\ZZ_0(5) \colonequals \XX_0(5) \rig \mu_2$. In particular, the groupoid of
  rational points $\ZZ_0(5)(\Q)$ is equivalent to the groupoid $\FF(\Q)$ with
  \begin{itemize}
  \item \textbf{Objects:} triples $(a,b,c) \in \Q\times \Q \times \Q - (0,0,0)$
    satisfying $a^2 + b^2 = c^4$.
  \item \textbf{Isomorphisms:} elements $\lambda \in \Q^\times$, sending
    $(a,b,c) \mapsto (\lambda^2a, \lambda^2 b, \lambda c)$.
  \end{itemize}
\end{example}

\subsection{Coordinate rings of stacky curves}
\label{sec:can-rings}

A \cdef{stacky curve} is a smooth proper geometrically connected
Deligne--Mumford stack, defined over a field. Some authors require that stacky curves contain a dense open
subscheme (see \cite[Chapter 5]{Voight&ZB22}), but our definition does not have this condition in order to 
include $\XX(1)$ as an example of stacky curve. Let $\XX$ be a stacky curve over
$\Q$, and let $\mathcal{L}$ be a line bundle on $\XX$. The \cdef{homogeneous
  coordinate ring relative to $\mathcal{L}$} is the graded ring
\begin{equation}
  \label{eq:R_L}
  R_\mathcal{L} \colonequals \bigoplus_{d \geq 0} H^0(\XX, \mathcal{L}^{\otimes d}).
\end{equation}
The degree $d$ piece of $R_{\mathcal{L}}$ is
$R_{\mathcal{L},d} \colonequals H^0(\XX,\mathcal{L}^{\otimes d})$. When
$\mathcal{L} \cong \mathcal{O}_\XX(D)$ where $D$ is a Cartier divisor, we
call $R_D \colonequals R_\mathcal{L}$.

\begin{example} \label{ex:X(1)_modular_forms} Recall from \Cref{ex:XX(1)} that
  $\XX(1) \cong \PProj \Q[u^4\x,u^6\y]$. In fact, this isomorphism identifies the
  polarizing line bundle $\mathcal{O}_{\XX}(1)$ with the Hodge bundle $\lambda$ on
  $\XX(1)$, so that $\XX(1) \cong R_\lambda$. Note that
  $R_\lambda = \Q[E_4,E_6]$ is the ring of modular forms on the modular curve
  $\XX(1)$ (with $\Q$-coefficients), where $E_4$ and $E_6$ are the Eisenstein
  series with constant term $1$ in their $q$-series expansion. Thus, $u^4\x$ is
  identified with $E_4$ and $u^6\y$ is identified with $E_6$ via the
  isomorphism $R_{\mathcal{O}_{\XX(1)}(1)} \cong R_\lambda$. 

  For our calculations, we use the standard normalization of the Weierstrass coefficients $A = -3E_4$ and $B = -2E_6$.
\end{example}

If $\mathcal{L}$ is a polarizing line bundle on a cyclotomic stacky curve
$\XX$, then $\XX$ can be recovered from $R_{\mathcal{L}}$ by the following
slight extension of \cite[Corollary 2.4.4]{Abramovich&Hasset11}; we thank 
Martin Olsson for sharing the statement while the authors attended the MRC conference.

\begin{lemma} \label{lemma:stacky_proj} Suppose that $\XX$ is a proper geometrically connected
  cyclotomic $\Q$-stack with a polarizing line bundle $\mathcal{L}$. Then
    \[
        \XX \cong \PProj R_{\mathcal{L}}.
    \]
\end{lemma}

\begin{proof}
  By the proof of \cite[Proposition 2.3.10]{Abramovich&Hasset11},
  $\XX \cong [P_{\mathcal L}/\Gm]$, where the $\Gm$-torsor
  $P_{\mathcal L} \colonequals \XX \times_{B\Gm} \Spec \Q$ on $\XX$ is defined by
  the classifying morphism $\XX \to B\Gm$ associated with $\mathcal L$. The proof
  of \cite[Corollary 2.4.4]{Abramovich&Hasset11} implies that $P_{\mathcal L}$ is an open subscheme of
  $\Spec R_{\mathcal L}$, so it suffices to show that
  $P_{\mathcal L} \cong \Spec R_{\mathcal L} - V(R_{\mathcal L,+})$ by the
  definition of $\PProj R_{\mathcal L}$.

  Denote by $c \colon \XX \to X$ the coarse map. Since $\mathcal{L}$ is a polarizing line
  bundle, there exists $e \in \Z_+$ and an ample line bundle $M$ on $X$ such that
  $\mathcal{L}^e \cong c^*M$. Connectedness of $\XX$ implies that
  $R_{\mathcal{L},0} = \Q = R_{M,0}$. Then $X \cong [P_M/\Gm]$ is a projective scheme with an
  ample line bundle $M$, so that $P_M \cong \Spec R_M - V(R_{M,+})$. By the proof
  of \cite[Corollary 2.4.4]{Abramovich&Hasset11} again, $\Spec R_{\mathcal L}$ is isomorphic to the relative
  normalization of $\Spec R_M$ in $P_{\mathcal L}$ via a morphism
  $P_{\mathcal L} \to \Spec R_M$ factoring through $P_M$. As the induced morphism
  $P_{\mathcal L} \to P_M$ is a finite surjection by loc. cit., the induced
  inclusion $P_{\mathcal L} \hookrightarrow \Spec R_{\mathcal L}$ identifies
  $P_{\mathcal L}$ with $\Spec R_{\mathcal L} - V(R_{\mathcal L,+})$, proving
  the desired assertion.
\end{proof}

For the remainder of this subsection, we apply \Cref{lemma:stacky_proj} to
describe a modular curve $\XX_0(5)$ as a stacky Proj. Let the $\Q$-stack
$\XX_0(5)$ be the modular curve parameterizing generalized elliptic curves over
$\Q$ with $\Gamma_0(5)$-structure, as in \cite{Cesnavicius17}. By \cite[IV.6.7]{Deligne&Rapoport73}, the proper DM stack
$\XX_0(5)$ is a smooth stacky curve, and it admits a projection
$J \colon \XX_0(5) \to \XX(1)$ which forgets the $\Gamma_0(5)$-structure on generalized
elliptic curves. Using $J$, we identify $\XX_0(5)$ with the stack $\GG$ from
\Cref{ex:GG}.

\begin{lemma}
  \label{lemma:XX05=GG}
  The stacks $\XX_0(5)$ and $\GG$ are isomorphic. Consequently, so are their
  rigidifications $\ZZ_0(5)$ and $\FF$.
\end{lemma}
\begin{proof}
  Recall from \Cref{ex:X(1)_modular_forms} that $\XX(1) \cong \PProj R_\lambda$, where
  $\lambda$ is the Hodge bundle on $\XX(1)$ and $R_\lambda$ is the ring of modular forms on
  $\XX(1)$. As $\lambda$ is a polarizing line bundle on $\XX(1)$ and
  $J \colon \XX_0(5) \to \XX(1)$ is a finite surjection, $\XX_0(5)$ is a geometrically
  connected proper cyclotomic stack equipped with a polarizing line bundle
  $J^*\lambda$. Thus, $\XX_0(5) \cong \PProj R_{J^*\lambda}$ by \Cref{lemma:stacky_proj}, and
  $R_{J^*\lambda}$ is the ring of modular forms on the modular curve
  $\XX_0(5)$. The computations in the file
  \href{https://github.com/sarangop1728/counting-5-isogenies/blob/main/abc-parametrization.m}{\texttt{abc-parametrization.m}}
  verify that the ring of modular forms $R_{J^*\lambda}$ is isomorphic, as graded
  $\Q$-algebras, to the ring $R$ from \Cref{ex:GG}. As $\GG \cong \PProj R$, we see
  that $\XX_0(5)$ and $\GG$ are isomorphic.
\end{proof}

\subsection{Geometric interpretation of the counting problem}
\label{sec:geometric-interpretation}

There are significant recent advances (\cite{Ellenberg&Satriano&ZB23, Darda&Yasuda22, 
Darda&Yasuda25}) in systematically defining heights of stacks as a direct 
generalization of height functions on projective varieties. 
In \cite[Definition 4.3]{Darda&Yasuda22}, a height of a smooth DM stack $\XX$ (with mild 
conditions as in \cite[Definition 2.1]{Darda&Yasuda22}) is defined by a choice of a
line bundle $\mathcal{L}$ on $\XX$ together with a bounded function (determined by a
raising datum in \cite[Definition 4.1]{Darda&Yasuda22}).
As a result, heights on $\XX$ corresponding to the same line bundle $\mathcal{L}$ are
$O(1)$-equivalent, as they only differ by the choice of bounded functions.

For the original counting problem (\Cref{thm:main-result}), the height function
comes from the line bundle $J^*\lambda^{12}$ on $\XX_0(5)$, where $J$ and $\lambda$
are as in \Cref{sec:can-rings}.
For the second counting problem (\Cref{thm:GG-count-main}), the height commes from 
$\mathcal{O}_\GG(2)$.
The statement and the proof of \Cref{lemma:XX05=GG} imply that $\mathcal{O}_\GG \cong J^*\lambda$ 
on $\GG \cong \XX_0(5)$; so, the $6^{\rm th}$-power of the second height is 
$O(1)$-equivalent to the first height.
This $O(1)$-equivalence turns out to be useful, as the lattice counting associated with
$N_5(H)$ defined by the first height is complicated, whereas the lattice counting associated
with $N_\GG(H^{1/6})$ is much simpler; in fact, the region coming from $N_\GG(H)$ is a disc.

However, the $O(1)$-equivalence does not guarantee that the counting functions 
have the same leading or lower-order terms, as predicted by the stacky Batyrev--Manin 
conjecture formulated in \cite[\S 9.2]{Darda&Yasuda22} (also see weak form in 
\cite[Conjecture 4.14]{Ellenberg&Satriano&ZB23}).
To obtain $N_5(H)$ from $N_G(H)$, we will use heights on rigidified stacks 
$\ZZ_0(5)$ and $\FF$ of $\XX_0(5)$ and $\GG$ respectively; \Cref{fig:geom-interpretation} 
summarizes the relations between the four stacks of interest.

\begin{figure}[ht]
  \centering
  \begin{tikzcd}
    \XX_0(5) \arrow[rr,"\cong"] \arrow[rd] \arrow[dd,"J"'] & & \GG
    \arrow[r,hook,"\iota"] \arrow[rd]\arrow[dd,dotted] & \PP(4,4,2)
    \arrow[rd] &\\
    & \ZZ_0(5) \arrow[rr,"\cong", near end, crossing over] & & \FF \arrow[dd] \arrow[r,hook] & \PP(2,2,1)\arrow[dd] \\
    \XX(1) \arrow[dd] \arrow[rr,dotted] & & \PP(4,6) \arrow[dd,dotted]& & \\
    & X_0(5)\arrow[rr,"\cong",near end, crossing over] \arrow[from=uu, crossing over] \arrow[ld] & & C_2 \arrow[ld] \arrow[r,hook] & \P^2 \\
    X(1) \arrow[rr, "\cong"] & & \Pone & &
  \end{tikzcd}
  \caption{This diagram summarizes the stacks considered in this paper and their respective coarse spaces. $C_2$ is the plane curve $\x^2 + \y^2 = \z^2$.}
  \label{fig:geom-interpretation}
\end{figure}

\section{Counting minimal integer solutions to \texorpdfstring{$x^2 + y^2 = z^4$}{x2+y2=z4}}
\label{sec:counting-ideals}
Recall that an integer triple $(a,b,c)$ is \cdef{\fft-minimal} if no
prime $\ell$ satisfies $\ell^4 \mid a,$ $\ell^4 \mid b,$ and $\ell^2\mid c$; equivalently, if the
$(4,4,2)$-weighted greatest common divisor of $(a,b,c)$ is one (see \cite{Beshaj&Gutierrez&Shaska20}). Since
we are concerned only with the generalized Fermat equation of signature
$(2,2,4)$, we will say that a triple of integers $(a,b,c)$ is a \cdef{Fermat
  triple} if it satisfies the equation $x^2 + y^2 = z^4$.

\begin{definition}
  \label{def:minimal-triple}
  We will say that a triple of integers $(a,b,c)$ is a \cdef{minimal Fermat
    triple} if the following conditions are satisfied:
  \begin{itemize}
   \item $a^2 + b^2 = c^4$,
   \item $(a,b,c)$ is $(4,4,2)$-minimal, and
   \item $c\neq0$.
   \end{itemize}
   The third condition is clearly redundant. We leave it to emphasize that $c$
   can be positive or negative. We let $\GG\ideal{\Q}$ denote the set of
   minimal Fermat triples, and we consider the counting function
   \begin{equation}
     \label{eq:N_G}
     N_\GG(H) \colonequals \#\brk{(a,b,c) \in \GG\ideal{\Q} : |c| \leq H}.
   \end{equation}
   \end{definition}

   This eccentric counting problem is intimately related to counting $5$-isogenies; indeed, $\GG\ideal{\Q}$ is the set of rational points on the stack $\GG$, which we know is isomorphic to $\XX_0(5)$ by \Cref{lemma:XX05=GG}.
   
   The main theorem of this section is the following.

\begin{theorem}
    \label{thm:GG-count}
    There exist explicitly computable constants $g_1, g_2, g_3 \in \R$ such that
    for every $\varepsilon > 0$, we have
    $$N_{\GG}(H) = g_1 H(\log H)^2 + g_2H(\log H) + g_3H + O_\varepsilon(H^{1/2+\varepsilon}),$$
    as $H \to \infty$. Furthermore, the constant $g_1$ is given by
    \begin{equation*}
      g_1  = \frac{48}{\pi^4}\prod_{p\, \equiv\, 1 \md 4} 
      \Bigl(1 - \frac{4}{(p+1)^2} \Bigr) > 0.
      \end{equation*}
\end{theorem}

\subsection{The rigidified count}
Observe that given a minimal Fermat triple $(a,b,c)$ and a square free integer
$e$, the (quadratic) \cdef{twist} $(e^2a,e^2b,ec)$ is another minimal Fermat triple. By
factoring out the $(2,2,1)$-greatest common divisor of $(a,b,c)$, and twisting
by $-1$ when $c <0$, we arrive at a unique minimal twist.

\begin{definition}
  An integer triple $(a,b,c)$ is a \cdef{twist minimal Fermat triple} if the
  following conditions are satisfied:
  \begin{itemize}
   \item $a^2 + b^2 = c^4$,
   \item $(a,b,c)$ is $(2,2,1)$-minimal, and
   \item $c>0$.
   \end{itemize}
   Let $\FF\ideal{\Q}$ be the set of twist minimal Fermat triples, and consider
   the following counting function
   \begin{equation}
     \label{eq:N_F}
     N_\FF(H) \colonequals \#\brk{(a,b,c) \in \FF\ideal{\Q} : c \leq H}.
   \end{equation}
\end{definition}

To prove \Cref{thm:GG-count}, we first obtain the asymptotics of $N_\FF(H)$ and
then estimate the number of twists of a given height to recover the asymptotics
of $N_\GG(H)$.

\begin{theorem}
    \label{thm:FF-count}
    There exist explicitly computable constants $f_1, f_2\in \R$ such that for
    every $\varepsilon > 0$, we have
    $$N_{\FF}(H) = f_1 H(\log H) + f_2H + O_\varepsilon(H^{1/2+\epsilon}),$$
    as $H \to \infty$. Furthermore, the constant $f_1$ is given by
    \begin{equation*}
      f_1 = \frac{\pi^2}{12} g_1  =\frac{4}{\pi^2} 
      \prod_{p\, \equiv\, 1 \md 4} 
      \left(1 - \frac{4}{(p+1)^2}\right) > 0.
      \end{equation*}
\end{theorem}

We start by recording some elementary observations about twist minimal Fermat
triples that will be needed in the proof of this theorem. These follow
immediately from the definitions, and Fermat's theorem on numbers representable
as the sum of two squares.

\begin{lemma}[Characterization of twist minimal triples]
  \label{lemma:twist-min-triples}
  Suppose that $(a,b,c)$ is a twist minimal Fermat triple. Then, the
  following are true:
  \begin{enumerate}[label=(\alph*)]
  \item $\gcd(a,b)$ is square free.
  \item If $p \mid \gcd(a,b)$, then $p \mid c$.
  \item If $p \mid c$, then $p \equiv 1 \md 4$.
  \item $c \equiv 1 \md 4$.
  \item $a$ and $b$ have distinct parities.
  \item If $a = 0$, then $b^2 = c^2 = 1$.
  \item If $b = 0$, then $a^2 = c^2 = 1$.
  \end{enumerate}
\end{lemma}

Define the \cdef{height zeta function} corresponding to $N_\FF(H)$ to be the
Dirichlet series
\begin{equation}
\label{eq:F}
  F(s) \colonequals \sum_{c = 1}^\infty \dfrac{f(c)}{c^s},
\end{equation}
where $f(c)$ is the \cdef{arithmetic height function} that counts the number of
twist minimal triples $(a',b',c')$ with $c' = c$. The following lemma describes
the function $f(c)$ in terms of the prime factorization of $c$.

\begin{lemma}
  \label{lemma:arithmetic-height-fun-f}
  The arithmetic height function $f(c)$ satisfies the following properties.
    \begin{enumerate}
    \item For a prime $\ell \not\equiv 1 \md 4$, we have $f(\ell^r) = 0$ for every positive
      integer $r$.
    \item For a prime $p \equiv 1 \md 4$, we have $f(p^r) = 16$ for every positive
      integer $r$.
    \item The arithmetic function $f(c)/4$ is multiplicative.
    \end{enumerate}
\end{lemma}

\begin{proof}
  Let $\FF_c$ denote the set of integral ideals $\mfa = (\alpha)$ in $\Z[i]$ such
  that $\alpha$ gives rise to a twist minimal Fermat triple $(\Re(\alpha), \Im(\alpha), c)$.
    \begin{enumerate}
    \item From \Cref{lemma:twist-min-triples}, we know that
      $\FF_\ell = \emptyset$ when $\ell \not\equiv 1 \md 4$.
    \item For such a prime $p$, write $p\Z[i] = \mfp\mfpb$ for the prime ideal
      factorization. From \Cref{lemma:twist-min-triples}, we deduce that
        \begin{equation*}
          \FF_{p^r} = \brk{\mfp^{4r}, \, \mfpb^{4r}, \, p\cdot\mfp^{4r-2}, \, p\cdot\mfpb^{4r-2}}.
        \end{equation*}
        It follows that $f(p^r) = 4(\#\FF_{p^r}) = 16$ for every $r\geq 1$.
      \item Indeed, $f(c)/4$ counts the number of twist minimal integral ideals. It
        is straightforward to see that twist minimality is a multiplicative
        condition for relatively prime ideals. The result follows from the
        multiplicativity of the ideal norm.
    \end{enumerate}
\end{proof}

\begin{proof}[Proof of {\Cref{thm:FF-count}}]
  Let $\chi_{4}$ denote the quadratic Dirichlet character modulo 4. The associated $L$-series has the Euler product expansion
\begin{equation*}
  L(s,\chi_4) = \prod_{p\equiv 1 \md 4}(1-p^{-s})^{-1}\prod_{q\equiv 3 \md 4}(1+q^{-s})^{-1}
  = \prod_{p} (1-\chi_{4}(p)p^{-s})^{-1}.
\end{equation*}    
Comparing the local Euler factors on the two sides shows that
 \begin{equation*}
   F(s) = 4\prod_{p \, \equiv\, 1 \md 4}\paren{\dfrac{1+3p^{-s}}{1-p^{-s}}}= \Big(\frac{2\zeta(s)L(s,\chi_4)}{(1+2^{-s})\zeta(2s)}\Big)^2\prod_{p \, \equiv\, 1 \md 4}\frac{(1+3p^{-s})(1-p^{-s})}{(1+p^{-s})^2},
 \end{equation*}
 where the Euler product converges for $\Re(s)>1/2$. We deduce the identity
 \begin{equation}
   \label{eq:F=xi.P}
   F(s) = \zeta(s)^2P(s),
 \end{equation}
 where
 \begin{equation}
   \label{eq:P}
   P(s) =  \Big(\frac{2L(s,\chi_4)}{(1+2^{-s})\zeta(2s)}\Big)^2\prod_{p \, \equiv\, 1 \md 4}\frac{(1+3p^{-s})(1-p^{-s})}{(1+p^{-s})^2}
 \end{equation}
 converges for $\Re(s)>1/2$. From \Cref{eq:F=xi.P}, we see that $F(s)$ has a
 unique pole at $s = 1$ of order $2$, and admits a meromorphic continuation to
 the half-plane $\Re(s) > 1/2$. The convexity bound for $\zeta(s)^2$ allows us
 to apply a standard Tauberian theorem (see \cite[Th\'eoréme
 A.1]{ChambertLoir&Tschinkel01}) to conclude the result. Recalling
 $L(1,\chi_4) = \pi/4$ and $\zeta(2)=\pi^2/6$, we have $2L(1,\chi_4)/\zeta(2) = 3/\pi$ and
 hence
 \begin{align*}
   f_1 &= \lim_{s\to 1}(s-1)^2F(s) = P(1)=
         \frac{4}{\pi^2} 
         \prod_{p\, \equiv\, 1 \md 4} 
         \Bigl(1 - \frac{4}{(p+1)^2} \Bigr).
 \end{align*}
\end{proof}

Languasco and Moree \cite{Languasco&Moree25} have obtained a precise
  numerical approximation of $f_1$ proceding as in Cohen \cite[\S
  10.3.6]{Cohen07}. They calculated
 \begin{equation}
   \label{f1decimals}    
   f_1 =
   0.3426610885510607694963830299360837190535240748255719116603071029\dotsc
 \end{equation}

\subsection{Proof of \texorpdfstring{\Cref{thm:GG-count}}{Theorem 3.2}}
\label{sec:proof-gg-count} As before, we aim to understand the analytic
properties of the height zeta function corresponding to $N_\GG(H)$, that is
\begin{equation*}
    G(s) \colonequals \sum_{n=1}^\infty \frac{g(n)}{n^{s}}, 
\end{equation*}
where $g(n)$ is the arithmetic height function that counts the number of
minimal Fermat triples $(a,b,c)$ with $|c| = n$.

Instead of directly analyzing the zeta function $G(s)$, we follow the strategy
employed in \cite[Theorem 5.1.4]{Molnar&Voight23} to leverage our
understanding of the rigidified zeta function $F(s)$.

\begin{theorem}
  The following statements hold:
  \begin{enumerate}[label=(\roman*)]
  \item \label{item:arithmetic-functions} $g(n) = 4(\mu^2 \star f)(n)$, where $\mu(n)$ is the M\"{o}bius function and
    $\star$ denotes convolution.
  \item \label{item:zeta-functions} $G(s) = 4\tfrac{\zeta(s)}{\zeta(2s)}F(s)$ in the half-plane $\Re(s) >
    1$.
  \item \label{item:analytic-properties} The function $G(s)$ admits meromorphic continuation to the half-plane
    $\Re(s) > 1/2$ with a triple pole at $s=1$ and no other singularities.
  \end{enumerate}
\end{theorem}

\begin{proof}
  For every square free $e \in \Z$, let
    \begin{equation*}
        g^{(e)}(n) \colonequals \#\brk{(a,b,c) \in \Z^3 : (e^2a, e^2b, ec) \in
          \GG\ideal{\Q} \mand |ec| = n}.
    \end{equation*}
    It follows from the definitions that
    \begin{equation}
        g^{(e)}(n) = 
        \begin{cases}
            2f(n/|e|), & \mif e \mid n, \\
            0, & \mif e \nmid n.
        \end{cases}
    \end{equation}
    Therefore,
    \begin{align*}
        g(n) = \sum_{\substack{e \in \Z \\ \square\text{ free}}} g^{(e)}(n) =
      4\sum_{\substack{e > 0 \\ e \mid n}} \mu^2(e)f(n/e) = 4(\mu^2\star f)(n),
    \end{align*}
    completing the proof of \Cref{item:arithmetic-functions}.
    \Cref{item:zeta-functions} follows directly form
    \Cref{item:arithmetic-functions}, and \Cref{item:analytic-properties}
    follows from the identity $F(s) = \zeta(s)^2P(s)$ and the proof of
    \Cref{thm:FF-count}.
  \end{proof}

  \begin{proof}[Proof of {\Cref{thm:GG-count}}]
    From the identity $G(s) = 4\zeta(s)F(s)/\zeta(2s) =
    4\zeta^3(s)P(s)/\zeta(2s)$, we can apply a Tauberian theorem once more to
    conclude the result. We have that the constant term is
    \begin{align*}
      g_1 &= \tfrac12\lim_{s\to 1}(s-1)^3G(s) \\
          &= 2\lim_{s\to 1}(s-1)^3\zeta(s)^3P(s)/\zeta(2s) = 2P(1)/\zeta(2) = 12P(1)/\pi^2.
    \end{align*}
  \end{proof}
    
\subsection{The ideal theoretic point of view}
\label{sec:counting-setup}
We end this section by shifting our perspective from integral triples to ideals
in $\Z[i]$. This is natural since Fermat triples are invariant under multiplication by
$i$. That is, if $(a,b,c)$ is a (twist minimal) Fermat
triple, then so are $(\pm a, \pm b, c)$ and $(\pm b, \pm a, c)$. The exposition becomes
clearer if we normalize by this symmetry.

Let $\II$ denote the multiplicative
monoid of (nonzero) integral ideals in $\Z[i]$. For every
$\mfa = (\alpha) \in \II$, we denote by
$N\mfa \colonequals \# \left(\Z[i]/\mfa \right) = \alpha\oalpha$ its norm. If $p \equiv 1 \md 4$
is a prime, we let $\mfp, \mfpb$, with $p\Z[i] = \mfp\mfpb$, denote the primes above. Similarly, if
$q \equiv 3 \md 4$ is prime, we let $q\Z[i] = \mfq$ denote the prime above it in $\Z[i]$.

Say that an ideal $\mfa \in \II$ is \cdef{(twist) minimal} if any generator
$\alpha$ gives rise to a (twist) minimal Fermat triple
$(\Re(\alpha), \Im(\alpha), c)$, for some value of $c = (N\mfa)^{1/4}$. Denote
by $\phi,\psi\colon \II \to \brk{0,1}$ the characteristic functions of the
property of being minimal, and twist minimal, respectively. Consider the
corresponding Dirichlet series:
\begin{equation}
  \label{eq:Lg-Lf}
  L(\phi,s) \colonequals \sum_{\mfa\in\II}\phi(\mfa)(N\mfa)^{-s}, \quad
  L(\psi,s) \colonequals \sum_{\mfa\in \II}\psi(\mfa)(N\mfa)^{-s}.
\end{equation}
Rearranging the sums and normalizing gives
\begin{align}
  L(\phi,s) &= \sum_{n=1}^\infty\paren{\sum_{N\mfa = n^4}\phi(\mfa)}n^{-4s} =
  \sum_{n=1}^\infty \tfrac18 g(n)n^{-4s} = \tfrac18 G(4s), \\
  \label{eq:L-F} L(\psi,s) &= \sum_{c=1}^\infty\paren{\sum_{N\mfa = c^4}\psi(\mfa)}c^{-4s} =
  \sum_{n=1}^\infty \tfrac14 f(n)n^{-4s} = \tfrac14 F(4s). 
\end{align}

In particular, we will consider the counting functions 
\begin{align}
  \label{eq:N-phi}
  N_\DD(\phi, H) &\colonequals \sum_{N\mfa \leq H^2} \phi(\mfa) = \tfrac18
                   N_\GG(H^{1/2}), \\
  \label{eq:N-psi}
  N_\DD(\psi, H) &\colonequals \sum_{N\mfa \leq H^2} \psi(\mfa) = \tfrac14 N_\FF(H^{1/2}).
\end{align}
This notation will be justified in the following section.

\section{Counting equidistributed Gaussian ideals in squareish regions}
\label{sec:counting-gaussian-ideals}

We consider the problem of counting ideals $\mathfrak{a} \subset \mathbb{Z}[i]$ 
with bounded norm that satisfy a certain property. However, our interest lies 
in counting these ideals with respect to a different height function. 
Geometrically, this corresponds to transforming a lattice point counting 
problem within a ball into that within a different region $\Omega$. Our goal is to 
formalize the intuition that if:
\begin{itemize}
    \item the region $\Omega$ is not too different from the ball, and 
    \item the ideals with this property have ``uniformly distributed angles'', 
\end{itemize}
then asymptotics should have analogous formulae.

\begin{figure}[htbp]
    \centering
    \begin{subfigure}{0.49\textwidth}
        \centering
        \includegraphics[width=\textwidth]{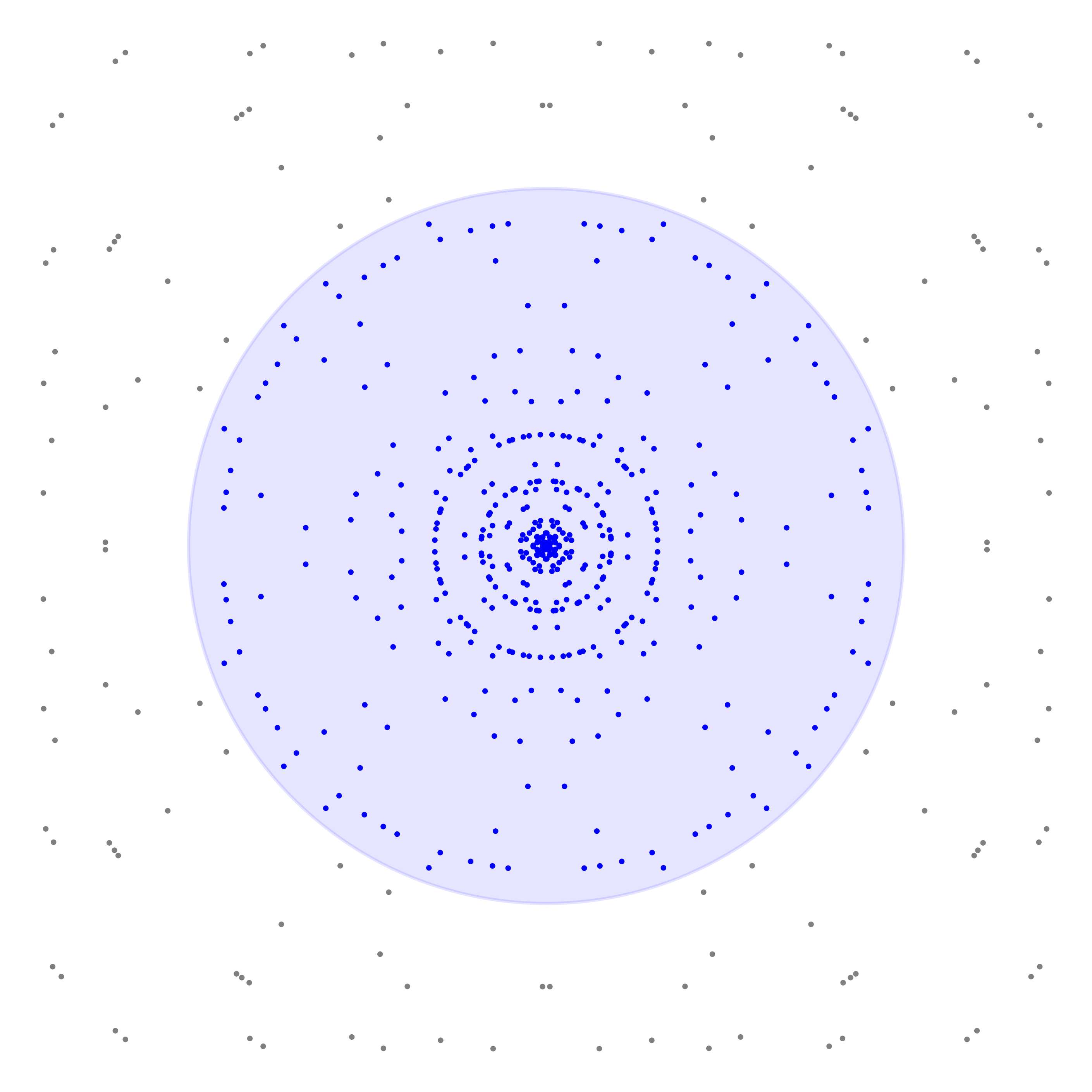}
        \caption{{\color{blue}$\Omega = \DD\colonequals \brk{z\in \C: |z|
            \leq 1}$.}}
        \label{fig:subfigure1}
    \end{subfigure}
    \hfill
    \begin{subfigure}{0.49\textwidth}
        \centering
        \includegraphics[width=\textwidth]{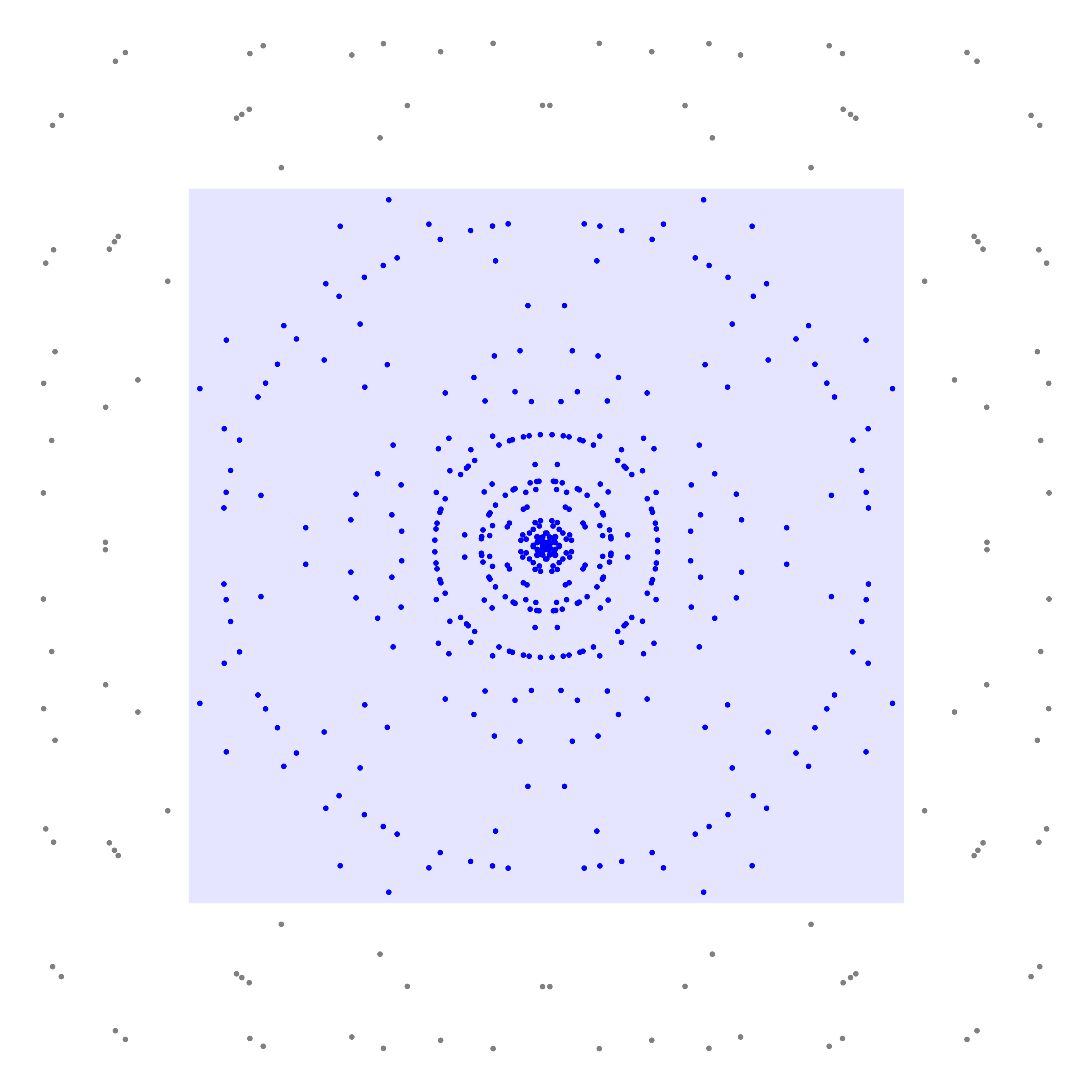}
        \caption{{\color{blue}$\Omega = \brk{x+iy \in \C : \max(|x|, |y|) \leq 1}$.}}
        \label{fig:subfigure2}
    \end{subfigure}
    \caption{The lattice points plotted above correspond to Gaussian integers
      $\alpha$ such that the ideal $\mfa = (\alpha)$ is twist minimal. The region $H\Omega$ is colored blue.}
    \label{fig:figure}
\end{figure}

\subsection{Setup}
\label{sec:analytic-setup}
Let $\Z[i]$ denote the ring of Gaussian integers, with
fraction field $\Q(i)$. Let $\II$
denote the multiplicative monoid of (nonzero) integral ideals in $\Z[i]$. For
every $\mfa = (\alpha) \in \II$, we denote by
$N\mfa \colonequals \# \left(\Z[i]/\mfa \right) = \alpha\oalpha$ its norm. Let $\DD \subset \C$ denote the
closed unit ball.

\begin{definition}\label{def:squareish}
  We say that a region $\Omega$ in $\C$ is \cdef{squareish} if it satisfies
  the following properties.
\begin{enumerate}
  \item[(R1)] $0 \in \Omega$.
  \item[(R2)] $\Omega$ is compact.
  \item[(R3)] There is a piecewise-smooth continuous function $\omega\colon \R \to \R_{\geq 0}$ of period $\pi/2$ such that
  \begin{equation*}
      \Omega = \brk{re^{i\theta} : r \in \R, \theta \in [0,2\pi], \text{ and } 0 \leq r \leq \omega(\theta)}.
  \end{equation*}
\end{enumerate}  
\end{definition}

Fix $s\in \C$, and consider the function $\omega(\theta)^{2s}$. This function
admits a Fourier expansion of the form
\begin{equation}
  \label{eq:fourier-expansion}
  \omega(\theta)^{2s} \colonequals \sum_{k \in 4\Z} I_{\Omega,k}(s)e^{ik\theta},
\end{equation}
where
\begin{equation}
  \label{eq:fourier-coef}
  I_{\Omega,k}(s) =
  \frac{1}{2\pi}\int_0^{2\pi}\omega(\theta)^{2s}e^{-ik\theta}\dd \theta = \frac{2}{\pi}\int_0^{\pi/2}\omega(\theta)^{2s}e^{-ik\theta}\dd \theta.
\end{equation}
Given $\mfa \in \II$, let $\theta_\mfa \in [0,\pi/2)$ be the argument of a
generator $(\alpha) = \mfa$ lying in the first quadrant. For every $k\in4\Z$
we have the \cdef{Hecke characters} $\xi_k\colon \II \to \C\unit$ given by
\begin{equation}
  \label{eq:xi-k}
  \xi_k(\mfa) \colonequals \paren{\dfrac{\alpha}{|\alpha|}}^{ik} = e^{ik\theta_\mfa},
\end{equation}
where $\alpha$ is any generator of the ideal $\mfa$. The final piece of data is
a function $\phi\colon \II \to \C$ defined on
ideals, and the corresponding \emph{twisted} Dirichlet series
\begin{equation}
  \label{eq:L-k}
L(s, \phi\otimes\xi_k) \colonequals \sum_{\mfa \in \II}
\phi(\mfa)\xi_k(\mfa)(N\mfa)^{-s}, \quad \mfor k \in 4\Z.
\end{equation}
For our applications, $\phi$ is the characteristic function of some property, and
an appropriate normalization of these Dirichlet series are \cdef{$L$-functions}
(in the sense of \cite[Section 5.1]{Iwaniec&Kowalski04}). In particular, (\ref{eq:L-k}) converges absolutely in the
half-plane $\Re(s) > \sigma_0>0$. When $k = 0$, we abbreviate
$L(s,\phi) \colonequals L(s,\phi\otimes\xi_0)$. Note that when
$\phi = \mathbf{1}$ is the constant function equal to one,
$L(s,\mathbf{1}) = \zeta_K(s)$ is the Dedekind zeta function of $K = \Q(i)$, and
$L(s,\mathbf{1}\otimes\xi_k) = L(s,\xi_k)$ are the Hecke $L$-functions.

\subsection{The main analytic lemma}
\label{sec:main-analytic-lemma}

Our goal is to understand the asymptotic main term in the counting function
\begin{equation}
  \label{eq:counting-function}
  N_\Omega(H;\phi) \colonequals \sum_{\mfa \in H\Omega}\phi(\mfa),
\end{equation}
where $H\Omega \colonequals \brk{Hz : z\in \Omega}$. The abuse of notation
$\mfa \in H\Omega$ means that any generator $\alpha$ of $\mfa$ is in
the intersection $\Z[i]\cap H\Omega$. Since the norm is a quadratic function, we
have that $\mfa \in H\Omega$ if and only if $N\mfa \leq \omega(\theta_\mfa)^2H^2$.

\begin{lemma}[Main analytic lemma]
  \label{lemma:main-analytic-lemma} Suppose that $\Omega\subset \Z$ is a
  squareish region, and that $\phi\colon \II \to \C$ is a
  function satisfying the following:
  \begin{enumerate}[label=(\roman*)]
  \item \label{it:MAL-reasonable-props} $L(s,\phi)$ has reasonable analytic properties:
    \begin{enumerate}[label=(\alph*)]
    \item $\phi(\mfa) \in \R_{>0}$ for all $\mfa \in \II$.
    \item There exists some $\sigma_0 \in \R_{>0}$ such that $L(s,\phi)$ converges absolutely in the half-plane $\Re(s) > \sigma_0 > 0$.
    \item $L(s,\phi)$ admits a meromorphic continuation to a half-plane $\Re(s) >
      \sigma_0 - \delta_0 > 0$. 
    \item In this domain, $L(s,\phi)$ has a unique pole. It is a simple pole of order $r = r(\phi) >
      0$, at $s=\sigma_0$. We denote
      \begin{equation*}
        \label{eq:theta}
        \Theta(\phi) \colonequals \lim_{s\to \sigma_0}(s-\sigma_0)^rL(s,\phi) > 0.
      \end{equation*}
      \item In the half-plane $\Re(s) > \sigma_0 - \delta_0$, we have the convexity
        bound
        \begin{equation*}
          \label{eq:convexity}
          |L(s,\phi)(s-\sigma_0)^r/s^r| = O(|1 + \Im(s)|^\kappa),
        \end{equation*}
        for some
        $\kappa = \kappa(\phi) > 0$.
    \end{enumerate}
      \item \label{item:equidist} For every $0 \neq k \in 4\Z$, there exist real valued functions $B_\phi(H,k)$ such that:
      \begin{itemize}
          \item $B_\phi(H,k) = o(H^{2\sigma_0}) $ as $H \to \infty$, 
          \item $B_\phi(H,k) = o(k^m)$ as $k \to \infty$ for some positive $m > 0$, and
          \item we have the bound
          \begin{equation*}
          \left| \sum_{N\mfa \leq H^2} \phi(\mfa)\xi_k(\mfa)\right| \leq
          B_\phi(H,k).
        \end{equation*}
      \end{itemize}
  \end{enumerate}
  Then, there exists a monic polynomial $P(T) = P_{\phi,\Omega}(T) \in \R[T]$
  of degree $r-1$ such that for every $0 < \delta < \delta_0$, the asymptotic formula
  \begin{equation*}
    \label{eq:asymptotic-omega}
    N_\Omega(H;\phi) = \dfrac{I_{\Omega,0}(\sigma_0)\Theta(\phi)2^{r-1}}{(r-1)!} H^{2\sigma_0}P(\log
    H) + O\left(\max\brk{\sum_{0 \neq k \in 4\Z} \frac{B_\phi(H,k)}{k^{m+2}}, H^{2(\sigma_0-\delta)} }\right)
  \end{equation*}
  holds as $H \to \infty$. Here, the implicit constant depends on $\delta, \Omega$, and $\phi$. 
\end{lemma}

Remembering the formula for area in terms of polar coordinates gives the following corollary.
\begin{corollary}
  In the situation of \Cref{lemma:main-analytic-lemma}, if $\sigma_0 = 1$, then
  we have
  \begin{equation*}
    N_\Omega(H;\phi) \sim \dfrac{\vol(\Omega)}{\pi}  N_\DD(H;\phi).
  \end{equation*}
\end{corollary}
\begin{proof}[Proof of \Cref{lemma:main-analytic-lemma}]
  Let $c \colonequals \sigma_0+1$. From Perron's integral, we have
  \begin{equation}
    \label{eq:perron}
    \dfrac{1}{2\pi i}\int_{c-i\infty}^{c +
      i\infty}\paren{\dfrac{\omega(\theta_\mfa)^2H^2}{N\mfa}}^s \dfrac{\dd
      s}{s} =
    \begin{cases}
      1, & \mif N\mfa < \omega(\theta_\mfa)^2H^2, \\
      \tfrac12, & \mif N\mfa = \omega(\theta_\mfa)^2H^2, \\
      0, & \mif N\mfa > \omega(\theta_\mfa)^2H^2. 
    \end{cases}
  \end{equation}
  for any fixed $\mfa \in \II$. Replace $\omega(\theta)^{2s}$ by its Fourier expansion
  (\ref{eq:fourier-expansion}). The smoothness hypothesis implies the superpolynomial decay of $I_{\Omega,k}(s)$. This permits exchanging the order of integration and
  summation. Summing over all $\mfa \in \II$ we obtain
\begin{equation}
  \label{eq:perron-sum}
  N_\Omega(H;\phi) = \dfrac{1}{2\pi i}\sum_{k\in
    4\Z}\int_{c-i\infty}^{c+i\infty}L(s,\phi\otimes\xi_k)I_{\Omega,k}(s)H^{2s}\dfrac{\dd
  s}{s},
\end{equation}
which holds for almost all $H > 0$. In particular, $H$ can be chosen so that
the Perron integral does not equal $\frac{1}{2}$ for all
$\mathfrak{a} \in \mathcal{I}$.

The proof proceeds in two steps. First, we show that the
dominant term in the sum is $k=0$ with negligible error
\begin{equation}
  \label{eq:claim-1}
  N_\Omega(\phi; H) = \dfrac{1}{2\pi i} \int_{c-i\infty}^{c+i\infty}L(s,\phi)I_{\Omega,0}(s)H^{2s}\dfrac{\dd
    s}{s} + O\left(\sum_{0 \neq k \in 4\Z} \frac{B_\phi(H,k)}{k^{m+2}} \right),
\end{equation}
for some positive integer $m$. Then, we show that
\begin{equation}
    \label{eq:claim-2}
  \dfrac{1}{2\pi i} \int_{c-i\infty}^{c+i\infty}L(s,\phi)I_{\Omega,0}(s)H^{2s}\dfrac{\dd
    s}{s} = \dfrac{I_{\Omega,0}(\sigma_0)\Theta(\phi)2^{r-1}}{(r-1)!} H^{2\sigma_0}P(\log
    H) + O(H^{2(\sigma_0-\delta)}).
\end{equation}
We will show these assertions in \Cref{lemma:claim-1} and \Cref{lemma:claim-2}.
When $\sigma_0 = 1$, the corollary follows from the area formula in polar coordinates:
\begin{equation*}
  \dfrac{\vol(\Omega)}{\vol(\DD)} = \frac{1}{\pi}\int_0^{2\pi}\frac12
  \omega(\theta)^2\dd\theta = I_{\Omega,0}(1).
\end{equation*}
\end{proof}

\begin{lemma}
  \label{lemma:claim-1} Assertion (\ref{eq:claim-1}) holds.
\end{lemma}
\begin{proof}
  Let $c \colonequals \sigma_0 + 1$. Starting from \Cref{eq:perron-sum}, we need to show that for some positive integer $m$,
  \begin{equation*}
    \dfrac{1}{2\pi i}\sum_{0\neq k\in
    4\Z}\int_{c-i\infty}^{c+i\infty}L(s,\phi\otimes\xi_k)I_{\Omega,k}(s)H^{2s}\dfrac{\dd
  s}{s} = O\left(\sum_{0 \neq k \in 4\Z} \frac{B_\phi(H,k)}{k^{m+2}} \right).
\end{equation*}
For every individual $0\neq k \in 4\Z$, \Cref{item:equidist} ensures that
there exists a function $B_\phi(H,k)$, depending only on $\phi$, uniformly
bounding the sums $\sum\phi(\mfa)\xi_k(\mfa)$. On the other hand, the superpolynomial decay of
$I_{\Omega,k}(s)$ gives us a constant $C_m(\Omega) > 0$, independent of
$k$, such that $|I_{\Omega,k}(s)| \leq C_m(\Omega)|/|k|^m$ for every
$m \in \Z_{>0}$. From the triangle inequality, we have that
\begin{align*}
  \left|\dfrac{1}{2\pi i} \int_{c-i\infty}^{c+i\infty}L(s,\phi\otimes\xi_k)I_{\Omega,k}(s)H^{2s}\dfrac{\dd
  s}{s}\right| &\leq \dfrac{C_m(\Omega)}{|k|^m}\left|\dfrac{1}{2\pi i} \int_{c-i\infty}^{c+i\infty}L(s,\phi\otimes\xi_k)H^{2s}\dfrac{\dd
  s}{s}\right|\\
  &\leq \dfrac{C_m(\Omega)}{|k|^m}\left|\sum_{N\mfa \leq
  H^2}\phi(\mfa)\xi_k(\mfa)\right| \leq \dfrac{C_m(\Omega)}{|k|^m} B_\phi(H,k).
\end{align*}
Taking $m$ sufficiently large, and summing over all $0\neq k \in 4\Z$, we
obtain the result.
\end{proof}

\begin{lemma}
  \label{lemma:claim-2} Assertion (\ref{eq:claim-2}) holds.
\end{lemma}
\begin{proof}
  Let $c \colonequals \sigma_0 +1$. The idea is the standard trick of shifting
  the vertical segment from $-c-iT$ to $c+iT$ to the vertical segment from
  $\sigma_0-\delta-iT$ to $\sigma_0-\delta+iT$, for some $0 < \delta < \delta_0$, so that we pick up the pole at
  $s=\sigma_0$. Then we can write the integral on the left hand side of
  \Cref{eq:claim-2} as
  \begin{equation}
    \label{eq:res+error}
    \Res_{s=\sigma_0}\left[L(s,\phi)I_{\Omega,0}(s)H^{2s}/s\right] + \frac{1}{2\pi
      i}\int_{\Gamma_{\delta, T}} L(s,\phi)I_{\Omega,0}(s)H^{2s}\frac{\dd s}{s}.
  \end{equation}
  Instead of repeating the proof of the Weiner--Ikehara Tauberian theorem (see for
  instance \cite[Appendice A]{ChambertLoir&Tschinkel01}), we explain how the
  extraneous factor $I(s) = I_{\Omega,0}(s)$ transforms the asymptotic.

  Using the fact that $I(s)$ is uniformly bounded in the strip
  $\sigma_0-\delta \leq \Re(s) \leq c$, we see that
  \begin{equation*}
    \frac{1}{2\pi
      i}\int_{\Gamma_{\delta, T}} L(s,\phi)I(s)H^{2s}\frac{\dd s}{s} = O\paren{\frac{1}{2\pi
      i}\int_{\Gamma_{\delta, T}} L(s,\phi)H^{2s}\frac{\dd s}{s}}.
\end{equation*}
Choosing an appropriate value of $T$ in terms of $H$, the proof of the
classical theorem gives that this term is $O(H^{2(\sigma_0-\delta)})$.

To calculate the residue, we compute the Laurent expansions of the terms around
$s = \sigma_0$ and calculate the coefficient of $(s-\sigma_0)^{-1}$ in their product. We have:
\begin{align*}
  L(s,\phi) &= \sum_{n=-r}^\infty \Theta_n(s-\sigma_0)^n, \\
  I(s) &= \sum_{a=0}^\infty \tfrac{I^{(a)}(\sigma_0)}{a!}(s-\sigma_0)^a, \\
  H^{2s} &= \sum_{b=0}^\infty \tfrac{H^{2\sigma_0}(2\log H)^b}{b!}(s-\sigma_0)^b, \\
  1/s &= \sum_{c=0}^\infty \tfrac{(-1)^c}{\sigma_0^{n+1}}(s-\sigma_0)^c, 
\end{align*}
It follows that the residue is given by the formula
\begin{equation}
  \label{eq:residue}
  \Res_{s=\sigma_0}\left[L(s,\phi)I(s)H^{2s}/s\right] = \sum_{n+a+b+c=-1} \Theta_n \tfrac{I^{(a)}(1)}{a!}\tfrac{H^22^b(\log H)^b}{b!}\tfrac{(-1)^c}{\sigma_0^{c+1}}.
\end{equation}
In particular, the main term occurs when $n = -r$, $b = r-1$, and $a = c = 0$, so
that the leading constant in the asymptotic is
\begin{equation}
  \label{eq:main-term}
  I_{\Omega,0}(\sigma_0)\cdot\dfrac{\Theta_{-r}2^{r-1}}{\sigma_0(r-1)!} > 0.
\end{equation}
\end{proof}

\section{Proof of Theorem 1.1}
\label{sec:proof}

We already know how to count minimal Fermat triples $(a,b,c)$ with $|c| \leq H$. We give a
bijective correspondence between such triples and 5-isogenies
$(E_{a,b,c},C_{a,b,c})$, up to isomorphism. Thus, we can write our counting
function as
\begin{equation*}
  N_5(H) = \#\brk{(a,b,c)\in \GG\ideal{\Q} : \Ht(E_{a,b,c}) \leq H},
\end{equation*}
where $\Ht(E_{a,b,c})$ is the naive height of
$E_{a,b,c}$ as in \Cref{eq:ec-height}. We have the compact region
\begin{equation}
  \label{eq:R5}
  \RR_5 \colonequals \brk{(x,y) \in \R^2 : H\paren{x,y} \leq 1}.
\end{equation}
The strategy is to write $N_5(H)$ as a rescaling of $N_{\RR_5}(H;\phi)$ for some
arithmetic function $\phi$, and then use the main analytic lemma
(\Cref{lemma:main-analytic-lemma}) to complete the proof. To do this, we need
to overcome two obstacles:
\begin{enumerate}
\item Firstly, the region $\RR_5$ is not squareish. To overcome this,
  we partition
  $\RR_5 = \RR_5^{(i)}\sqcup \RR_5^{(ii)} \sqcup \RR_5^{(iii)}\sqcup
  \RR_5^{(iv)}$ into its quadrants, and rotate each component to obtain four
  squareish regions $\Omega^{(i)},\Omega^{(ii)},\Omega^{(iii)}, \Omega^{(iv)}$. Further details can
  be found in \Cref{sec:rigidified_count}.
\item Secondly, even when $(a,b,c)$ is a minimal triple, our chosen Weierstrass model for $E_{a,b,c}$ is not always a minimal one. We overcome this issue by carefully quantifying how far
  is $E_{a,b,c}$ from being minimal, and counting these finitely many degenerate cases separately.
  Further details can be found in \Cref{sec:parametrization} and
  \Cref{sec:proof-thm:main-result}.
\end{enumerate}

\subsection{The parametrization}
\label{sec:parametrization}
Given a minimal Fermat triple $(a,b,c)$, consider the curve whose Weierstrass
equation is given by
\begin{equation}
  \label{eq:Eabc}
  \y^2 = \x^3 + A(a,b,c)\x + B(a,b,c),
\end{equation}
where
\begin{align}
  \label{eq:A}
  A(a,b,c) &\colonequals -6(123a + 114b + 125c^2) \\
  &= -(2\cdot 3^2\cdot 41)a - (2^2\cdot 3^2 \cdot 19)b - (2\cdot 3\cdot 5^3)c^2, \notag\\
  \label{eq:B}
  B(a,b,c) &\colonequals 8c(2502a + 261b + 2500c^2) \\
  &= (2^4\cdot 3^2\cdot 139)ac + (2^3\cdot 3^2\cdot 29)bc + (2^5\cdot 5^4)c^3. \notag
\end{align}
When the discriminant
$\Delta(a,b,c) \colonequals -16(4A(a,b,c)^3 + 27B(a,b,c)^2)$ is not zero, the
smooth projective model of \Cref{eq:Eabc} is an elliptic curve $E_{a,b,c}$
defined over $\Q$. Furthermore, this elliptic curve admits a rational
$5$-isogeny $C_{a,b,c} \subseteq E_{a,b,c}(\Qbar)[5]$. Indeed, the $5$-division
polynomial (see \cite[Exercise 3.7]{Silverman86}) $\psi_5(T;E_{a,b,c}) \in \Z[a,b,c][T]$ has a quadratic factor 
\begin{equation}
  \label{eq:habc}
  h(a,b,c;T) \colonequals 5T^2 - 100cT - 2106a - 1750c^2 + 792b.
\end{equation}
The cyclic group $C_{a,b,c}$ is defined as the group generated by the points $P
\in E_{a,b,c}(\Qbar)[5]$ whose $\x$-coordinates are the roots of $h(a,b,c;T)$.

We will denote by $\XX\ideal{\Q}$ the set of isomorphism classes of the
groupoid $\XX(\Q)$. We have $\XX_0(5)\ideal{\Q} = X_0(5)(\Q)$. As a corollary of
\Cref{lemma:XX05=GG} we have the following parametrization of isomorphism classes
of $5$-isogenies in terms of minimal Fermat triples. Given a 5-isogeny
$(E,C)$, we denote its $\Q$-isomorphism class by $[E,C]$.
\begin{lemma}
  \label{lemma:parametrization}
  The equivalence of \Cref{lemma:XX05=GG} induces a bijection
  $\Phi\colon\GG\ideal{\Q} \to X_0(5)(\Q)$. Under this map, the cusps
  $\Phi(Q_1),\Phi(Q_2) \in X_0(5)(\Q)$ correspond to the minimal
  triples $$Q_1\colonequals (-1,0,1), \quad Q_2\colonequals (-528,220,25).$$ In
  particular, $\Phi$ restricts to an isomorphism
  \begin{equation}
    \label{eq:Psi}
    \Phi\colon \GG\ideal{\Q}-\brk{Q_1,Q_2}\longrightarrow Y_0(5)(\Q), \quad
    (a,b,c) \mapsto [E_{a,b,c},C_{a,b,c}].
  \end{equation}
\end{lemma}
The identical parametrization of twist minimal isomorphism classes of
$5$-isogenies in terms of twist minimal Fermat triples also holds. Given a
$5$-isogeny $(E,C)$ defined over $\Q$, we denote its $\Qbar$-isomorphism class
by $\ideal{E,C}$.
\begin{lemma}
\label{lemma:parametrization_twist}
The equivalence of \Cref{lemma:XX05=GG} induces a bijection
$\Psi\colon \FF\ideal{\Q} \to \ZZ_0(5)\ideal{\Q}$. Likewise, the cusps
$\Psi(Q_1),\Psi(Q_2) \in \ZZ_0(5)(\Q)$ correspond to the twist minimal Fermat
triples $$Q_1\colonequals (-1,0,1), \quad Q_2\colonequals (-528,220,25).$$ In
particular, $\Psi$ restricts to an isomorphism
  \begin{equation}
      \label{eq:Psi_rigid}
      \Psi\colon \FF\ideal{\Q}-\brk{Q_1,Q_2} \longrightarrow (\ZZ_0(5)\ideal{\Q} - \brk{\Psi(Q_1),\Psi(Q_2)}), \quad (a,b,c) \mapsto \ideal{E_{a,b,c},C_{a,b,c}}.
  \end{equation}
\end{lemma}

Observe that the representatives $(E_{a,b,c},C_{a,b,c})$ of the isomorphism
classes in \Cref{lemma:parametrization} and \Cref{lemma:parametrization_twist}
do not need to correspond to a minimal Weierstrass equation.

\begin{definition}
  A Weierstrass equation of the elliptic curve $E\colon \y^2 = \x^3 + A\x + B$ is
  \cdef{twist minimal} if no prime $p$ satisfies $p^2 \mid A$ and $p^3 \mid B$.
\end{definition}

Since the naive height (\ref{eq:ec-height}) is calculated from the coefficients of a minimal model, it is necessary to
understand the twist minimal triples $(a,b,c)$ giving rise to Weierstrass
equations $E_{a,b,c}$ that are not twist minimal.

\begin{definition}
  The \cdef{twist minimality defect} of a Fermat triple $(a,b,c)$, denoted by
  $\twmind(E_{a,b,c})$, is
  \begin{equation*}
    \twmind(E_{a,b,c}) \colonequals \max\{e \in \Z_{>0} \; : \; e^2 \mid A(a,b,c) \text{ and } e^3 \mid B(a,b,c)\}.
  \end{equation*}
  In the notation of \cite{Molnar&Voight23}, this is precisely the \emph{twist
  minimality defect} of the Weierstrass equation $E_{a,b,c}$.
  We say that a twist minimal Fermat triple $(a,b,c)$ is \cdef{exceptional} if
  $\twmind(E_{a,b,c}) > 1$.
\end{definition}

We classify the exceptional twist minimal Fermat triples.
\begin{proposition}
\label{prop:exceptional-triples}
    Let $(a,b,c)$ be a twist minimal Fermat triple. 
    \begin{enumerate}
        \item The set of all possible values of $\twmind(E_{a,b,c})$ is
        \begin{equation*}
            \TMD \colonequals \{1,2,5,10,25,50,125,250\}.
        \end{equation*}
      \item $2 \mid \twmind(E_{a,b,c})$ if and only if $2 \mid b$. Moreover, the
        $2$-adic valuations of $A(a,b,c), B(a,b,c),$ and $\twmind(E_{a,b,c})$ can be found in \Cref{table:val_2}.
      \item $5 \mid \twmind(E_{a,b,c})$ if and only if $\alpha = a +ib$ for any one of
        $\alpha \in \mathbb{Z}[i]$ provided in \Cref{table:val_5}.
    \end{enumerate}
  \end{proposition}
  
  \begin{table}[ht]
    \centering \setlength{\arrayrulewidth}{0.2mm} \setlength{\tabcolsep}{5pt}
    \renewcommand{\arraystretch}{1.2}
    \begin{tabular}{|c|c|c|}
      \hline
      \rowcolor{headercolor}
      $\ord_2(A)$ & $\ord_2(B)$ & $\ord_2(\twmind)$ \\ \hline
      $1$ & $3$ & $0$ \\ \hline
      $\geq 2$ & $4$ & $1$ \\ \hline
    \end{tabular}
    \caption{$2$-adic valuations of $A(a,b,c)$, $B(a,b,c)$, and
      $\twmind(E_{a,b,c})$.}
    \label{table:val_2}
  \end{table}

  \begin{table}[ht]
    \centering \setlength{\arrayrulewidth}{0.2mm} \setlength{\tabcolsep}{5pt}
    \renewcommand{\arraystretch}{1.2}
    \begin{tabular}{|c|c|c|c|}
      \hline
      \rowcolor{headercolor}
      $\alpha = a + ib \in \Z[i]$          & $\ord_5(A)$ & $\ord_5(B)$ & $\ord_5(\twmind)$ \\\hline
      $(1-2i)^4 \beta$       & $4$           & $5$           & $1$                       \\ \hline
      $5(1-2i)^2 \beta$      & $3$           & $4$           & $1$                       \\ \hline
      $(1-2i)^{4k} \beta$, $k \geq 2$    & $5$           & $k+7$       & $2$                       \\ \hline
      $5(1-2i)^{6} \beta$ & $6$           & $9$       & $3$                       \\ \hline
      $5(1-2i)^{4k-2} \beta$, $k \geq 3$ & $6$           & $k+8$       & $3$                       \\ \hline
    \end{tabular}
    \caption{$5$-adic valuations of $A(a,b,c)$, $B(a,b,c)$, and
      $\twmind(E_{a,b,c})$. Here, $\beta \in \Z[i]$ is any element of norm coprime
      to $5$.}
    \label{table:val_5}
  \end{table}
  
\begin{proof}
  We break up the proof of the proposition into a number of lemmas in modular
  arithmetic. All the Fermat triples below are assumed to be twist minimal.

  The first two statements follow by combining the statements of Lemmas
  \ref{lemma:exceptional-triples}, \ref{lemma:4-nmid-twerr},
  \ref{lemma:4-nmid-twerr-2}, and \ref{lemma:625-nmid-twerr}. The last
  statement follows from \Cref{lemma:characterization-5-exceptional-triples}.
\end{proof}
\begin{lemma}
  \label{lemma:p-mid-c-and-pp-mid-A}
  If a prime $p$ satisfies $p \mid c$ and $p^2 \mid A(a,b,c),$ then $p=5$.
\end{lemma}
\begin{proof}
  From \Cref{lemma:twist-min-triples}, every prime $p$ dividing $c$
  satisfies $p \equiv 1 \md 4$. Thus, $p \neq 2,3$. Assume that $p \neq 5$. Reducing \Cref{eq:A} modulo $p^2$
  gives $123a \equiv -114b \md p^2$. This implies that
  \begin{align*}
    (123c^2)^2 &= (123a)^2 + (123b)^2 \\
               &\equiv (-114b)^2 + (123b)^2 \md p^2 \\
               &\equiv (114^2 + 123^2)b^2 \md p^2 \equiv 3^2\cdot 5^5\cdot b^2 \equiv 0 \md p^2 .
  \end{align*}

It follows that $p \mid b$, and therefore $p \mid a$ as well. Let $a'
\colonequals a/p$, $b' \colonequals b/p$ and $c'\colonequals c/p$. We have that
$(a')^2 + (b')^2 = p^2(c')^4$ and $123a' + 114b' + 125p(c')^2 \equiv 0 \md p$.
Once more, it follows that $123a' \equiv -114b' \md p$, and
  \begin{align*}
    (123a')^2 + (123b')^2 &\equiv (114^2 + 123^2)(b')^2 \md p \\
                       &\equiv 3^2\cdot 5^5\cdot (b')^2 \equiv 0 \md p .
  \end{align*}
  This allows us to conclude that $p \mid b'$, and therefore $p \mid a'$ as well. This contradicts the
  twist minimality of the triple $(a,b,c)$.

  To see that the case $p=5$ indeed occurs, one can take $(a,b,c) = (-527,
  -336, 25)$. This triple is twist minimal, since $\gcd(a,b)=1$, and $A(a,b,c)
  = 2^4\cdot 3 \cdot 5^5$.
\end{proof}

\begin{lemma}
  If $p$ is a prime divisor of
  $\twmind(E_{a,b,c})$, then $p \in \brk{2,5}$.
\end{lemma}
\begin{proof}
  By definition, $p \mid \twmind(E_{a,b,c})$ if and only if
  $p^2 \mid A \colonequals A(a,b,c)$ and $p^3 \mid B \colonequals B(a,b,c)$. If $p$ also
  divided $c$, then \Cref{lemma:p-mid-c-and-pp-mid-A} implies that
  $p = 5$. For this reason, we assume that $p$ does not divide $c$.

  We use the equations $A \equiv 0 \md p^2$ and $B \equiv 0 \md p^2$ to find
  congruence relations between $a$ and $b$ modulo $p^2$. We find integer
  linear combinations of $cA$ and $B$ that allow us to cancel the terms with
  $a$ or $b$ in them, obtaining:
  \begin{align*}
    1112cA + 41B &= -(14000c^3 + 675000bc) = -2^35^3c(14c^2 + 675b) \equiv 0 \md
    p^2, \\
    58ca + 19B &= 337500ac + 336500c^3 = 2^25^3c(673c^2+675a) \equiv 0 \md p^2.
  \end{align*}
  If $p \neq \brk{2,5}$, then $(14c^2 + 675b) \equiv 0 \md
    p^2$ and $(673c^2+675a) \equiv 0 \md p^2$. But in this case, we can reduce
    the equation $(675a)^2 + (675b)^2 = 675^2c^4$ modulo $p^2$ to obtain
    \begin{equation*}
      5^6\cdot 29 = 14^2 + 673^2 \equiv 675^2 = 3^6\cdot 5^4  \md p^2,
    \end{equation*}
    producing a contradiction. We conclude that in this case $p \in \brk{2,5}$,
    and the result follows.
\end{proof}

\begin{lemma}
  \label{lemma:exceptional-triples}
  The following assertions hold.
  \begin{enumerate}[label=(\alph*)]
  \item \label{item:2-exceptional} $2 \mid \twmind(E_{a,b,c})$ if and only if
    $2 \mid b$.
  \item \label{item:5-exceptional} $5 \mid \twmind(E_{a,b,c})$ if and only if
    $a \equiv 7b \md 25$ and $5 \mid c$.
  \end{enumerate}
  We say that $(a,b,c)$ is \cdef{2-exceptional} when \ref{item:2-exceptional}
  holds, and \cdef{5-exceptional} when \ref{item:5-exceptional} holds.
\end{lemma}
\begin{proof}
  \hfill
  \begin{enumerate}[label=(\alph*)]
  \item If $2 \mid \twmind(E_{a,b,c})$, then $4 \mid A$. This implies that $a$ and
    $c$ have the same parity. Since $c \equiv 1 \md 4$, $a$ must be odd. But
    since $a$ and $b$ have opposite parity, we conclude that $2 \mid b$.
    Conversely, suppose that $2 \mid b$. This implies that $a$ is odd. Since $8
    \mid B$, we only need to show that $4 \mid A$. But this is visibly
    true once we know that $a$ is odd.
  \item If $5 \mid \twmind(E_{a,b,c})$, then $A \equiv 0 \md 25$. This implies that
    $a \equiv 7b \md 25$. From this congruence, we see that $c^4 = a^2 +
    b^2 = 50b^2 \equiv 0 \md 25$, so that $5 \mid c$. Conversely, the
    congruence $a \equiv 7b \md 25$ implies that $A \equiv 0 \md 25$ and $B/(8c)
    \equiv 0 \md 25$, which is enough to conclude that $5 \mid \twmind(E_{a,b,c})$.
  \end{enumerate}
\end{proof}

\begin{lemma}
  \label{lemma:4-nmid-twerr}
  The quantity $\twmind(E_{a,b,c})$ is not divisible by $4$.
\end{lemma}
\begin{proof}
  Suppose that $4 \mid \twmind(E_{a,b,c})$. By definition, this means that $4^2 \mid
  A$ and $4^3 \mid B$. From \Cref{eq:A,eq:B}, we deduce that
  \begin{align*}
    123a + 114b + 125c^2 &\equiv 0 \md 8, \\
    2502a + 261b + 2500c^2 &\equiv 0 \md 8.
  \end{align*}
  Since $c \equiv 1 \md 4$ (\Cref{lemma:twist-min-triples}), we have that
  $c^2 \equiv 1 \md 8$, and the above congruences simplify to:
    \begin{align*}
    3a + 2b + 5 &\equiv 0 \md 8, \\
    6a + 5b + 4 &\equiv 0 \md 8.
    \end{align*}
    These imply that $b \equiv 6 \md 8$, and $a^2 + b^2 = c^4$ implies that
    $a^2 \equiv 5 \md 8$. But this contradicts the fact that $5$ is not a
    square modulo $8$.
\end{proof}

\begin{lemma}
    \label{lemma:4-nmid-twerr-2}
    If $\twmind(E_{a,b,c})$ is divisible by $2$, then $\ord_2(A(a,b,c)) \geq 2$ and $\ord_2(B(a,b,c)) = 4$. Otherwise, $\ord_2(A(a,b,c)) = 1$ and $\ord_2(B(a,b,c)) = 3$.
\end{lemma}
\begin{proof}
    Suppose $2 \mid \twmind(E_{a,b,c})$. Using $a^2 + b^2 = c^4$ mod $16$, one obtains that $4 \mid b$, $a \equiv 1, 7, 9, 15 \text{ mod } 16$, and $c^2 \equiv 1, 9 \text{ mod } 16$. It follows that $\text{ord}_2(A(a,b,c)) \geq 2$ and $\ord_2(B(a,b,c)) = 4$. The case for $2 \nmid \twmind(E_{a,b,c})$ follows from the fact that $b$ and $c$ are odd.
\end{proof}

\begin{lemma} \label{lemma:625-nmid-twerr}
    The quantity $\twmind(E_{a,b,c})$ is not divisible by $5^4$.
\end{lemma}
\begin{proof}
    Suppose that $5^4 \mid \twmind(E_{a,b,c})$. By definition, this implies that $5^8 \mid A$ and $5^{12} \mid B$. Using \Cref{eq:A,eq:B}, we deduce that
    \begin{align*}
        123a + 114b + 125c^2 &\equiv 0 \text{ mod } 5^8 \\
        c(2502a + 261b + 2500c^2) &\equiv 0 \text{ mod } 5^{12}.
    \end{align*}
    We consider four cases depending on the 5-valuation of $c$.

    \medskip \textbf{Case 1}: Suppose $\ord_5(c) \geq 3$. Then,
    $41 a + 38b \equiv 0 \md 5^8$. This congruence simplifies to
    $a \equiv 323932b \md 5^8$. Using the equation $a^2 + b^2 = c^4$, we obtain
    \begin{equation*}
        a^2 + b^2 \equiv 300000a^2 = 2^5 \cdot 3 \cdot 5^5 a^2 \equiv 0 \md 5^8.
    \end{equation*}
    This implies that $a \equiv b \equiv 0 \text{ mod } 5^2$, contradicting
    the twist minimality of $(a,b,c)$.

    \medskip 
    \textbf{Case 2}: Suppose that $\ord_5(c) = 2$. We obtain the
    linear system
    \begin{align*}
        41a + 38b &\equiv 0 \md 5^7, \\
        278a + 29b &\equiv 0 \md 5^7.
    \end{align*}
    Solving for $b$ in the first equation and substituting into the second one,
    we arrive at
    \begin{equation*}
        18750a = 2 \cdot 3 \cdot 5^5 a \equiv 0 \md 5^7.
    \end{equation*}
    This implies that $a \equiv b \equiv 0 \md 5^2$, contradicting
    the twist minimality of $(a,b,c)$.

    \medskip \textbf{Case 3}: Suppose that $\ord_5(c) = 1$. Then
    $41a + 38b \equiv 0 \md 5^5$. Solving for $b$ we obtain
    $b \equiv 1068a \md5^5$, which implies
    \begin{equation*}
        a^2 + b^2 \equiv 3125 a^2 \equiv 5^5 a^2 \equiv 0 \equiv c^4 \text{ mod } 5^5.
    \end{equation*}
    This implies that $5^2 \mid c$, a contradiction.

    \medskip \textbf{Case 4}: Suppose that $\ord_5(c) = 0$. Then \Cref{eq:B}
    implies $278a + 29b \equiv 0 \text{ mod } 5^4$. Solving for $b$, we obtain
    $b \equiv 497 \md 5^4$, which implies
    \begin{equation*}
       c^4 = a^2 + b^2 \equiv 625 a^2 \equiv 5^4 a^2 \equiv 0 \md5^4.
    \end{equation*}
    This implies that $5 \mid c$, a contradiction.
\end{proof}

\begin{lemma} \label{lemma:characterization-5-exceptional-triples}
  Let $(a,b,c)$ be a $5$-exceptional triple, and let
  $\alpha = a+bi \in \Z[i]$. Then, we can write $\alpha = \delta\cdot\beta$
  where $\beta$ and $\delta$ are coprime twist minimal Gaussian integers, and
  $\delta$ is given as in \Cref{tab:factors-5-exceptional-triples}.
\end{lemma}
\begin{table}[ht]
  \centering
  \setlength{\arrayrulewidth}{0.2mm}
  \setlength{\tabcolsep}{5pt}
  \renewcommand{\arraystretch}{1.2}
  \begin{tabular}{|c|c|c|}
    \hline
    \rowcolor{headercolor}
    $\ord_5(\twmind(E_{a,b,c}))$ & $\delta$ & $k$              \\ \hline
    $1$   & $(1 - 2i)^4$ & -                           \\ \hline
    $1$   & $5(1 - 2i)^2$ & -                          \\ \hline
    $2$   & $(1 - 2i)^{4k}$ & $\geq 2$                        \\ \hline
    $3$   & $5(1 - 2i)^{4k-2}$  & $\geq 2$                    \\ \hline
    \end{tabular}
    \caption{Gaussian factors of $5$-exceptional triples.}
    \label{tab:factors-5-exceptional-triples}
\end{table}
\begin{proof}
  Since $(a,b,c)$ is $5$-exceptional, we know that $\ord_5(c) = k \geq 1$.
  Moreover, we can write $\alpha = \delta \cdot \beta$ so that $\beta$ is
  coprime to both $1\pm 2i$. This implies that $\delta \in \brk{(1\pm 2i)^{4k},
    5(1\pm 2i)^{4k-2}}$. On the other hand, observe that
  \begin{align}
   \label{eq:(1+2i)^5}
    (1+2i)^5  \alpha &= (41-38i)(a+bi) = (41a + 38b) + i(41b - 38a), \\
    \label{eq:(1+2i)^7}
        (1+2i)^7  \alpha &= (29+278i)(a+bi) = (29a -278b) + i(278a + 29b).
  \end{align}
  Taking the norm and comparing the parity of $\ord_5(\cdot)$ on the two sides of the
  equation, we deduce that
  \begin{align*}
    \ord_5(41a + 38b) &= \ord_5(41b - 38a) = \ord_5((1+2i)^5\alpha), \\
    \ord_5(278a + 29b) &= \ord_5(29a -278b) = \ord_5((1+2i)^7\alpha).
  \end{align*}
  Furthermore, since $5$ divides $A$ and $B$, we must have that
  $\delta \in \brk{(1- 2i)^{4k}, 5(1-2i)^{4k-2}}$. Taking $\ord_5(\cdot)$ on
  \Cref{eq:A,eq:B} we obtain
  \begin{align}
    \label{eq:ord-5-A}
    \ord_5(A(a,b,c)) &\geq \min\brk{\ord_5((1+2i)^5\delta), 3+2k}, \\
    \label{eq:ord-5-B}
    \ord_5(B(a,b,c)) &\geq k + \min\brk{\ord_5((1+2i)^7\delta), 4+2k},
  \end{align}
  with equality when the entries of the minimum functions are distinct.
\end{proof}

\subsection{The rigidified count}
\label{sec:rigidified_count}
We adopt the strategy of the proof of Theorem \ref{thm:GG-count} and apply the main analytic
lemma (Lemma \ref{lemma:main-analytic-lemma}) to obtain an explicit asymptotic formula for the counting functions
\begin{equation*}
  N_5^{\tw}(H,t) \colonequals \#\{(a,b,c) \in \FF\ideal{\Q} : \twmind(E_{a,b,c}) = t, \, \twHt(E_{a,b,c}) \leq H\}.
\end{equation*}
Observe that we are counting elliptic curves up to quadratic twists. Explicitly, our height function is
\begin{equation*}
  \twHt(E_{a,b,c}) \colonequals \dfrac{\max\brk{4|A(a,b,c)|^3, 27B(a,b,c)^2}}{\twmind(E_{a,b,c})^6}.
\end{equation*}

\begin{theorem}
\label{thm:main-twerr-rigid}
For each $t \in \TMD \colonequals \brk{1,2,5,10,25,50,125,250}$, there exist explicitly
computable constants $\hat{C}_{5,t}, \hat{C'}_{5,t} \in \mathbb{R}$ and a number
$c \in (0,1)$ such that
    \begin{equation*}
      N_5^{\tw}(H,t) = \hat{C}_{5,t} H^{1/6} (\log H) + \hat{C'}_{5,t} H^{1/6} + O(H^{1/6} \cdot (\log H)^{-1+c})
    \end{equation*}
    as $H \to \infty$. The constants $\hat{C}_{5,t}$ are given by:
    \begin{equation*}
      \hat{C}_{5,t} \colonequals \frac{f_1}{\pi} V_5\cdot \begin{cases}
        \frac{t}{128}, &\text{ if } t \in \{1,2\}, \\
        \frac{t}{240}, &\text{ if } t \in \{5,10\}, \\
        \frac{t}{1920}, &\text{ if } t \in \{25,50,125,250\},
        \end{cases} 
      \end{equation*}
      where $f_1$ is the constant defined in Theorem \ref{thm:FF-count}, and $V_5$ is an explicit constant given by \Cref{eq:V_5}.
\end{theorem}

Given $t \in \TMD$, we define the \cdef{twist minimality defect height zeta function} corresponding to $N_5^{\tw}(H,t)$ to be the Dirichlet series
\begin{equation}
    F_t(s) \colonequals \sum_{c=1}^\infty \frac{f_t(c)}{c^s},
\end{equation}
where we denote by $f_{t}(c)$ the \cdef{twist minimality defect arithmetic height function} that counts the number of twist minimal triples $(a',b',c')$ with $c' = c$ and $\twmind(a',b',c') = t$. The following lemma is an analogue of Lemma \ref{lemma:arithmetic-height-fun-f}.
\begin{lemma} \label{lemma:arithmetic-height-fun-f-mathfrakd}
    The arithmetic height function $f_{t}(c)$ satisfies the following properties.
    \begin{enumerate}
        \item If $p \not\equiv 1 \text{ mod } 4$, then $f_{t}(p^k) = 0$ for every positive integer $k$.
        \item Suppose $t \in \{1,2\}$.
        \begin{itemize}
            \item If $p \equiv 1 \text{ mod } 4$ and $p \neq 5$, then $f_{t}(p^k) = 8$ for every positive integer $k$.
            \item If $p = 5$, then $f_{t}(5^k) = 4$ for every positive integer $k$.
            \item The arithmetic function $f_{t}(c)/2$ is multiplicative, but not completely multiplicative.
        \end{itemize}
        \item Suppose $t \in \{5,10\}$.
        \begin{itemize}
            \item Given any integer $c$, if $\ord_5(c) \neq 1$ then $f_t(c) = 0$.
            \item We have $f_{t}(5) = 4$ and $f_{t}(5^k) = 0$ for every positive integer $k \geq 2$.
            \item For any $c$ corpime to $5$, we have $f_t(5c) = 4 f_{\frac{t}{5}}(c).$
        \end{itemize}
        \item Suppose $t \in \{25,50,125,250\}$.
        \begin{itemize}
            \item Given any integer $c$, if $\ord_5(c) \leq 1$, then $f_t(c) = 0$.
            \item We have $f_{t}(5) = 0$, and $f_{t}(5^k) = 2$ for every positive integer $k \geq 2$.
            \item For any $c$ coprime to $5$ and $k \geq 2$, we have $f_t(5^k c) = 2 f_{\frac{t}{25}}(c)$ if $t \in \{25,50\}$, and $f_t(5^k c) = 2 f_{\frac{t}{125}}(c)$ if $t \in \{125,250\}$.
        \end{itemize}
    \end{enumerate}
\end{lemma}
\begin{proof}
    We recall from \Cref{lemma:exceptional-triples} that $2 \mid \twmind(E_{a,b,c})$ if and only if $b$ is even. Because $(a,b,c)$ is twist minimal, $c$ is odd. This implies that one of $a$ and $b$ must be even. Without loss of generality, assume $b$ is even. Multiplication by $\pm i$ flips the role of $a$ and $b$. Statements (1) and (2) (except for the case where $p = 5$) follow from adapting the technique of the proof shown in \Cref{lemma:arithmetic-height-fun-f}, in particular by using multiplicativity of the ideal norm. To check the statement for $p = 5$, we use \Cref{lemma:characterization-5-exceptional-triples}. Suppose that $\mfa \subset \mathbb{Z}[i]$ is an ideal such that $5 \nmid \twmind(E_{a,b,c})$ and $N\mfa = 5^{4k}$. There are two such ideals, namely the ones generated by $(1+2i)^{4k}$ and $5 (1+2i)^{4k-2}$. 
    
    To check statements (3) and (4), we also use \Cref{lemma:characterization-5-exceptional-triples}. If $\twmind(E_{a,b,c}) = 5$ or $10$, then there are only two such ideals $\mfa \subset \mathbb{Z}[i]$, generated by $(1-2i)^4$ and $5 (1-2i)^2$: both ideals have norm $N\mfa = 5^4$. If $\twmind(E_{a,b,c}) = 25$ or $50$, then for each $k \geq 2$ such that $N\mfa  = 5^{4k}$, there is only one ideal $\mfa \subset \mathbb{Z}[i]$ satisfying the aforementioned conditions: it is the ideal generated by $(1-2i)^{4k}$. Lastly, if $\twmind(E_{a,b,c}) = 125$ or $250$, then for each $k \geq 2$ such that $N\mfa  = 5^{4k}$, there is only one ideal $\mfa \subset \mathbb{Z}[i]$ satisfying the aforementioned conditions: it is the ideal generated by $5 (1-2i)^{4k-2}$. To relate the arithmetic functions $f_t$ with $f_{\frac{t}{25}}$ or $f_{\frac{t}{125}}$, we use the fact that any ideal $\mfa$ such that $5 \mid \twmind(E_{a,b,c})$ must be of form $\delta \cdot \beta$ where $\delta$ has norm divisible by $5$, and $\beta$ is coprime to $5$.
\end{proof}

To prove Theorem \ref{thm:main-twerr-rigid}, we check whether the two conditions of the main analytic lemma are satisfied. Given a squareish region $\Omega \subset \C$ and $t \in \TMD$, we define $\psi_t: \mathcal{I} \to \mathbb{C}$ to be the indicator function of those ideals $\mfa = (\alpha)$ giving rise to a twist minimal triple $(a,b,c)$ with $\twmind(E_{a,b,c}) = t$. We denote by
\begin{equation}
    \label{eq:L(phi_t,s)}
    L(\psi_t,s) \colonequals \sum_{\mfa \in \II} \psi_t(\mfa)(N\mfa)^{-s}
\end{equation}
the corresponding Dirichlet series, as in \Cref{sec:analytic-setup}. Recall the definition of $L(\psi,s)$ in \Cref{eq:Lg-Lf}.
\begin{lemma}
\label{lemma:MAL-condition1}
The Dirichlet series $L(\psi_t,s)$ satisfy the reasonable analytic properties in \Cref{it:MAL-reasonable-props} of \Cref{lemma:main-analytic-lemma}, with $\sigma_0 = \frac{1}{4}$, $\delta_0 = \frac{1}{8}$, and $r = 2$. Moreover, we have

\begin{equation*}
    L(\psi_t, s) = \begin{cases}
        \paren{\frac{1 + 5^{-4s}}{1 + 3 \cdot 5^{-4s}}}L(\psi,s), & \quad \mif t\in\brk{1,2}, \\
        \frac{4}{5^{4s}} \paren{\frac{1 - 5^{-4s}}{1 + 3 \cdot 5^{-4s}}}L(\psi,s), & \quad \mif t \in \brk{5,10}, \\
        2 \paren{\frac{5^{-8s}}{1 + 3 \cdot 5^{-4s}}}L(\psi,s), & \quad \mif t \in \{25,50,125,250\}.
    \end{cases}
\end{equation*}
\end{lemma}
\begin{proof}
     Suppose $t \in \{1,2\}$. From \Cref{lemma:arithmetic-height-fun-f} and \Cref{lemma:arithmetic-height-fun-f-mathfrakd}, we have that $f(c) = 2f_t(c)$ as long as $c$ is coprime to $5$. This implies that for any $t \in \{1,2\}$, the height zeta functions $F_t(s)$ satisfy:
\begin{equation*}
        F_t(s) = \frac{1}{2}\paren{\frac{1 + 5^{-s}}{1 + 3 \cdot 5^{-s}}}F(s).
\end{equation*}
Now suppose that $t \in \{5,10\}$. Then we have
\begin{align*}
    F_t(s) &= \sum_{c=1}^{\infty} \frac{f_t(c)}{c^s} \\
    &= \sum_{\substack{c = 1 \\ \ord_5(c) = 1}}^\infty \frac{4 f_{\frac{t}{5}}(\frac{c}{5})}{c^s} = \frac{4}{5^s} \sum_{\substack{m = 1 \\ 5 \, \nmid \, m}}^\infty \frac{f_{\frac{t}{5}}(m)}{m^s} = \frac{2}{5^s} \paren{\frac{1 - 5^{-s}}{1 + 3 \cdot 5^{-s}}}F(s).
\end{align*}
Lastly, suppose that $t \in \{25,50,125,250\}$. Let $\nu$ momentarily denote the $5$-adic valuation $\ord_5$. Then we have
\begin{align*}
\begin{split}
    F_t(s) &= \sum_{c=1}^{\infty} \frac{f_t(c)}{c^s} = \sum_{k=2}^\infty \sum_{\substack{c = 1 \\ \nu(c) = k}}^\infty \frac{2 f_{\frac{t}{5^{\nu(t)}}}(\frac{c}{5^k})}{c^s} = \sum_{k=2}^\infty \frac{2}{5^{ks}} \sum_{\substack{m = 1 \\ 5 \, \nmid \, m}} \frac{f_{\frac{t}{5^{\nu(t)}}}(m)}{m^s} \\
    &= \sum_{k=2}^\infty \frac{2}{5^{ks}} \cdot \frac{1}{2} F(s) \cdot \frac{1 - 5^{-s}}{1 + 3 \cdot 5^{-s}} = \paren{\frac{5^{-2s}}{1 + 3 \cdot 5^{-s}}}F(s).
\end{split}
\end{align*}

Hence, the analytic properties of $F_t(s)$ are identical to those of $F(s)$, which are obtained in the proof of \Cref{thm:FF-count}. The result follows from the identities (\ref{eq:L-F}) and
\begin{equation}
    \label{eq:Lt-Ft}
    L(\psi_t,s) = \tfrac{1}{2} F_t(4s).
\end{equation}
\end{proof}

\begin{lemma}
\label{lemma:MAL-condition2}
For all $t \in \TMD$, the Dirichlet series $L(\psi_t,s)$ satisfy the hypothesis \Cref{item:equidist} of \Cref{lemma:main-analytic-lemma} with
\begin{equation*}
          B_{\psi_t}(H,k) \colonequals
          \frac{H^{1/2}}{(\log H)^{1-c}} \cdot \log k,
\end{equation*}
where $c = \frac{3\sqrt{3}}{4\pi} \in (0,1)$.
\end{lemma}
\begin{proof}
    The lemma follows from adapting the proof of angular equidstribution of Gaussian integers, as demonstrated in \cite[Section 2]{EH99}. Without loss of generality, we assume that $t = 1$, and abbreviate $\psi_1 = \psi$. The cases for other values of $t$ follow analogously, for some $c \in (0, \frac{3\sqrt{3}}{4\pi}]$.

    Given an ideal $\mathfrak{a} \in \mathcal{I}$, denote by $\overline{\mathfrak{a}}$ its complex conjugate. Given a prime $p \equiv 1 \text{ mod } 4$, we denote by $\theta_{p}$ the argument of a generator of the prime ideal of norm $p$ lying in the first octant. Since $\theta_\mathfrak{a} = \frac{\pi}{2} - \theta_{\overline{\mathfrak{a}}}$ and $\psi(\mathfrak{a}) = \psi(\overline{\mathfrak{a}})$, one can check that if $N\mathfrak{a} = p^{\ell}$ for some prime $p \equiv 1 \text{ mod } 4$ and a positive integer $\ell$, then
    \begin{equation*}
        \psi(\mathfrak{a}) \xi_k(\mathfrak{a}) + \psi(\overline{\mathfrak{a}}) \xi_k(\overline{\mathfrak{a}}) = \begin{cases}
            2 \psi(\mathfrak{a}) \cos(k \ell \theta_p), &\text{ if } p \nmid \mathfrak{a}, \\
            2 \psi(\mathfrak{a}) \cos(k(\ell-2) \theta_p), &\text{ if } p \parallel \mathfrak{a}, \\
            0, &\text{ otherwise }.
        \end{cases}
    \end{equation*}
    Hence, we have for each $p$ and $k$,
    \begin{align*}
        \sum_{\substack{\mathfrak{a} \in \mathcal{I} \\ N \mfa = p^\ell}} \psi(\mfa) \xi_k(\mfa) &= 2 \psi(\mfa) (\cos(k\ell \theta_p) + \cos(k(\ell-2) \theta_p)).
    \end{align*}
    Using the fact that $\psi(\mfa) \xi_k(\mfa)$ is a multiplicative function, $\psi(\mfa) = 0$ if the norm of $\mfa$ is not a $4$-th power, and $\psi(\mfa) \leq 4$ for any ideal $\mfa$ of norm a power of $5^4$, we obtain:
    \begin{align*}
        \left| \sum_{N\mfa \leq H^2} \frac{\psi(\mfa)\xi_k(\mfa)}{N \mfa}\right| &\leq 8 \cdot \prod_{\substack{p \equiv 1 \text{ mod } 4 \\ 13 \leq p \leq H^{1/2}}} \left(1 + \sum_{\ell=1}^\infty \frac{|\cos(k \ell \theta_p) + \cos(k (\ell - 2) \theta_p)|}{p^\ell} \right).
    \end{align*}
    We now focus on understanding the following expression:
    \begin{equation} \label{eq:thm:equid_2}
        \Sigma(H) \colonequals \sum_{\substack{p \equiv 1 \text{ mod } 4 \\ p \leq H^{1/2}}} |\cos(4k \theta_p) + \cos(2k \theta_p)|.
    \end{equation}
    The period of the trigonometric function $y = \cos (4k \theta) + \cos(2k \theta)$ is equal to $\frac{\pi}{k}$. When restricted to $I = [-\frac{\pi}{2k}, \frac{\pi}{2k}]$, the function is non-negative over the interval $I_2 = [-\frac{\pi}{6k}, \frac{\pi}{6k}]$, and non-positive over the intervals $I_1 = [-\frac{\pi}{2k}, -\frac{\pi}{6k}]$ and $I_3 = [\frac{\pi}{6k}, \frac{\pi}{2k}]$. Note that at the boundaries of the intervals $I_1,I_2$, and $I_3$, we have $\cos(4k \theta) + \cos(2k \theta) = 0$. Equation (\ref{eq:thm:equid_2}) can hence be rewritten as
    \begin{equation*}
        \Sigma(H) = \frac{k}{4} \sum_{\substack{p \equiv 1 \text{ mod } 4 \\ p \leq H^2}} \sum_{i=1}^3 (-1)^{i}  \left( \sum_{\theta_p \in I_i} \int_{\theta_p}^{\bar e_{i}} 4k \sin(4k \vartheta) + 2k \sin(2k\vartheta) \dd\vartheta \right)
    \end{equation*}
    where $\bar e_i \colonequals \text{sup}\{x : x \in I_i\}$. Note that the extra term $\frac{k}{4}$ originates from the fact that the period of the function $y = \cos (4k \theta) + \cos(2k \theta)$ is equal to $\frac{\pi}{k}$, and that $\theta_p$ lies in the first octant of the plane. By Fubini's theorem, we obtain
    \begin{equation*}
        \Sigma(H) = \frac{k}{4} \sum_{i=1}^3 (-1)^{i} \int_{I_i} \left[ \sum_{\substack{p \equiv 1 \text{ mod } 4 \\ p \leq H^{1/2} \\ \underline{e}_i < \theta_p \leq \vartheta}} 1 \right]\cdot (4k \sin(4k \vartheta) + 2k \sin(2k\vartheta)) \dd\vartheta 
    \end{equation*}
    where $\underline{e}_i \colonequals \text{inf}\{x : x \in I_i\}$. 

    To estimate the integrals, we use a theorem of Kubilius.
    \begin{theorem}[{\cite[Theorem 8]{Ku50}}]
        The number of prime ideals $\mfp \subset \Z[i]$ with $N\mfp \equiv 1 \mod 4$ in the sector $0 \leq \alpha \leq \arg(\mfp) \leq \beta \leq \frac{\pi}{2}$, and $\Nm \mfp \leq u$ is equal to
        \begin{equation*}
            \frac{2}{\pi}(\beta - \alpha) \int_2^u \frac{\dd v}{\log v} + O\paren{u\cdot\exp(-b \sqrt{\log u})},
        \end{equation*}
        as $u \to \infty$, where $b$ is a positive absolute constant.
    \end{theorem}
    Kubilius's theorem allows us to rewrite the desired sum as
    \begin{align*}
    \begin{split}
        \Sigma(H) &= \frac{k}{4} \left[ \frac{2}{\pi} \int_2^{H^{1/2}} \frac{\dd v}{\log v} \right]  \sum_{i=1}^3 (-1)^{i} \int_{I_i} (\vartheta - \underline{e}_i)(4k \sin(4k \vartheta) + 2k \sin(2k \vartheta)) \dd\vartheta  \\
        & \quad + O\paren{kH^{1/2}\cdot\exp(-b \sqrt{\log H})} \\
        &= c \int_2^{H^{1/2}} \frac{\dd v}{\log v} + O\paren{kH^{1/2}\cdot \exp(-b \sqrt{\log H})},
    \end{split}
    \end{align*}
    where
    \begin{equation*}
        c = \frac{k}{2\pi} \cdot \left(\frac{3\sqrt{3}}{8k} + \frac{3\sqrt{3}}{4k} + \frac{3\sqrt{3}}{8k} \right) = \frac{3\sqrt{3}}{4\pi}.
    \end{equation*}
    Using Abel's partial summation formula, for any $2 \leq \omega \leq H^{1/2}$, we obtain
    \begin{align*}
        \sum_{\substack{p \equiv 1 \text{ mod } 4 \\ p \leq H^{1/2}}} \frac{1}{p} |\cos(4k \theta_p) + \cos(2k \theta_p)| \leq \frac{1}{2} \log \log \omega + c \log \left( \frac{\log H^{1/2}}{\log \omega} \right) + O(1) + O\paren{k \cdot e^{-b \sqrt{\log \omega}}}.
    \end{align*}
    We may choose $\omega = (\frac{1}{b} \cdot \log k)^2$ to get a uniform bound for $0 \neq k \in 4 \mathbb{N}$:
    \begin{equation*}
        \sum_{\substack{p \equiv 1 \text{ mod } 4 \\ p \leq H^{1/2}}} \frac{1}{p} |\cos(4k \theta_p) + \cos(2k \theta_p)| \leq c \log \log H + (1-2c) \log \log k + O(1).
    \end{equation*}
    Hence, we obtain
    \begin{align*}
        \left| \sum_{N\mfa \leq H^2} \frac{\psi(\mfa)\xi_k(\mfa)}{N\mfa}\right| &\ll (\log H)^c \cdot (\log k)^{1-2c},
    \end{align*}
    which in turn implies 
    \begin{align*}
        \left| \sum_{N\mfa \leq H^2} \psi(\mfa)\xi_k(\mfa) \right| &\ll \frac{H^{1/2}}{(\log H)^{1-c}} (\log k)^{1-2c} \ll \frac{H^{1/2}}{(\log H)^{1-c}} (\log k).
    \end{align*}
\end{proof}

With the two conditions for the \Cref{lemma:main-analytic-lemma} satisfied, we are able to prove the desired estimates for $N_5^{\tw}(H,t)$.
\begin{proof}[Proof of \Cref{thm:main-twerr-rigid}]
    Define the compact region
    \begin{equation*}
    \Omega_5 \colonequals \left\{ (x,y) \in \mathbb{R}^2 \colon H \left(x,y\right) \leq 1 \right\},
    \end{equation*}
    where 
    \begin{equation*}
        H(x,y) = \max\brk{27|A(x,y,\sqrt[4]{x^2 + y^2})|^3, 4B\paren{x,y,\sqrt[4]{x^2 + y^2}}^2}.
    \end{equation*}
    for the polynomials $A$ and $B$ given in \Cref{eq:A} and \Cref{eq:B}.
    Using the change of variables $x = r^2\cos \theta, y = r^2\sin \theta$, we may rewrite $\Omega_5$ as the region bounded by the radius function
    \begin{align}
    \label{eq:omega-5}
        \omega_5(\theta) &\colonequals \tfrac{1}{12}\max\brk{\tfrac{1}{2} |123 \cos \theta + 114 \sin \theta + 125|^3, |2502 \cos \theta + 261 \sin \theta + 2500|^2}^{-\frac{1}{3}}.
    \end{align}
    We note that the above function is well defined because the arguments of the maximum function are strictly positive. 
    
    Note that
    \begin{equation}
        \label{eq:V_5}
        V_5 \colonequals \frac{1}{2} \int_0^{2\pi} \sqrt{\omega_5(\theta)} \dd \theta.
    \end{equation}
    
    We subdivide $\Omega_5$ into four sub-regions $\Omega_5^{(i)}, \Omega_5^{(ii)}, \Omega_5^{(iii)}, $ and $\Omega_5^{(iv)}$ by intersecting it with each quadrant.
    Denote by $\Omega^{(i)}, \Omega^{(ii)}, \Omega^{(iii)}, \Omega^{(iv)}$ the squareish domains (\Cref{def:squareish}) obtained by rotating the four regions above by multiples of $\pi/2$. We note that
    \begin{equation*}
        V_5 = \frac{1}{4} \left(I_{\Omega^{(i)},0}(1/4) + I_{\Omega^{(ii)},0}(1/4) + I_{\Omega^{(iii)},0}(1/4) + I_{\Omega^{(iv)},0}(1/4) \right).
    \end{equation*}

    For each $t \in \TMD$, Lemma \ref{lemma:MAL-condition1} implies that
    \begin{equation}
    \label{eqn:theta-psi-t}
        \Theta(\psi_t) = 
        \begin{cases}
            \frac{3}{256} f_1 &\text{ if } t \in \{1,2\}, \\
            \frac{1}{160} f_1 &\text{ if } t \in \{5,10\}, \\
            \frac{1}{1280} f_1 &\text{ if } t \in \{25,50,125,250\},
        \end{cases}
    \end{equation}
    where we recall that $f_1 = \lim_{s\to 1}(s-1)^2F(s)$, and $\Theta(\psi_t)$ is as in \Cref{eq:theta} in \Cref{lemma:main-analytic-lemma}.

    By Lemma \ref{lemma:main-analytic-lemma}, there exists monic polynomials $P_t(T)$ of degree one such that as $H \to \infty$, we have
    \begin{equation*}
        N_{\Omega_5}(H ; \psi_t) = \frac{V_5}{\pi}  f_1 \cdot \begin{cases}
            \frac{3}{128}  \cdot H^{1/2} P_t(\log H) + O(H^{1/2} (\log H)^{-1+c}) &\text{ if } t \in \{1,2\}, \\ 
            \frac{1}{80} \cdot H^{1/2} P_t(\log H) + O(H^{1/2} (\log H)^{-1+c}) &\text{ if } t \in \{5,10\}, \\ 
            \frac{1}{640} \cdot H^{1/2} P_t(\log H) + O(H^{1/2} (\log H)^{-1+c}) &\text{ if } t \in \{25,50,125,250\}.
        \end{cases}
    \end{equation*}
    It remains to determine which bound $H$ one needs to use in order to recover the counting function $N_5^{\tw}(H,t)$. Let $\mfa \in \II$ be an ideal associated with the twist minimal triple $(a,b,c)$. Once more, the change of coordinates $a = c^2 \cos \theta_\mfa$ and $b = c^2 \sin \theta_\mfa$ allows us to write the twist height function as
    \begin{align*}
    \twHt(E_{a,b,c}) &= \frac{12^3  c^6  \max\left( \frac{1}{2}|123 \cos \theta_\mfa + 114 \sin \theta_\mfa + 125|^3, |2502 \cos \theta_\mfa + 261 \sin \theta_\mfa + 2500|^2 \right)}{\twmind(E_{a,b,c})^6} \\
    &= c^6 \cdot \frac{1}{\omega_5(\theta_\mfa)^3} \cdot \frac{1}{\twmind(E_{a,b,c})^6}\\
    &= \frac{1}{\omega_5(\theta_\mfa)^{3}} \cdot \frac{c^6}{t^6}
    \end{align*}
    Since $N\mfa = c^4$ and the condition $\mfa \in H\Omega_5$ is equivalent to $N\mfa \leq \omega_5(\theta_\mfa)^2H^2$, this implies that 
    \begin{equation*}
        N_5^{\tw}(H,t) = N_{\Omega_5} \left({H^{1/3} \cdot t^2} ; \psi_t \right),
    \end{equation*}
    from which the statement of the theorem follows.
\end{proof}

\subsection{Proof of Theorem \ref{thm:main-result}}
\label{sec:proof-thm:main-result}
The proof of Theorem \ref{thm:main-result} will be analogous to the idea of the proof presented in \Cref{sec:proof-gg-count}. However, there is a technical subtlety where the Weierstrass model of the elliptic curve $E_{a,b,c}$ obtained from a minimal triple $(a,b,c)$ may not be of minimal form. Using the twist minimal defect of $(a,b,c)$ introduced in the previous section, we precisely quantify the non minimality of $E_{a,b,c}$.

Given $t \in \TMD$ and a square free $e \in \Z$, let
\begin{equation}
    g_t^{(e)}(n) \colonequals \#\{(a,b,c) \in \Z^3 : (e^2 a, e^2 b, ec) \in \GG \langle \Q \rangle, \, \twmind(E_{a,b,c}) = t, \text{ and } |ec| = n\}
\end{equation}
Given a Weierstrass model of an elliptic curve $E\colon y^2 = x^3 + Ax + B$ with $A,B \in \Z$, we denote by $\mind(E)$ the \cdef{minimality defect} of $E$ defined as
\begin{equation}
    \mind(E) \colonequals \max\{e \in \Z : e^4 \mid A \text{ and } e^6 \mid B\}.
\end{equation}
Observe that in general
\begin{equation*}
  \Ht(E_{a,b,c}) = \dfrac{\max\brk{4|A(a,b,c)|^3, 27B(a,b,c)^2}}{\mind(E_{a,b,c})^{12}}.
\end{equation*}
\begin{lemma}
\label{lemma:md-triples}
    Let $(a,b,c) \in \Z^3$ be a twist minimal Fermat triple with twist minimality defect $\twmind(E_{a,b,c}) = t$, and let $e$ be a square free integer.
    \begin{itemize}
    \item The $2$-adic valuation of $\mind(E_{e^2a, e^2b, ec})$ can be calculated from \Cref{tab:2-adic-vals}.
    \begin{table}[ht]
      \centering
      \setlength{\arrayrulewidth}{0.2mm}
      \setlength{\tabcolsep}{5pt}
      \renewcommand{\arraystretch}{1.2}
      \caption{$2$-adic valuations of minimality defects.}
      \label{tab:2-adic-vals}
        \begin{tabular}{|c|c|c|c|c|}
          \hline
          \rowcolor{headercolor}
          $\ord_2(\twmind(E_{a,b,c}))$ & $\ord_2(e)$ & $\ord_2(A(e^2a,e^2b,ec))$ & $\ord_2(B(e^2a,e^2b,ec))$ & $\ord_2(\mind(E_{e^2a,e^2b,ec}))$ \\
          \hline
          $0$ & $0$ & $1$ & $3$ & $0$ \\
          $0$ & $1$ & $3$ & $6$ & $0$ \\
          $1$ & $0$ & $\geq 2$ & $4$ & $0$ \\
          $1$ & $1$ & $\geq 4$ & $7$ & $1$ \\
          \hline
        \end{tabular}
    \end{table}
    \item The $5$-adic valuation of $\mind(E_{e^2a, e^2b, ec})$ can be calculated from \Cref{tab:5-adic-vals}.
    \begin{table}[ht]
      \centering
      \setlength{\arrayrulewidth}{0.2mm}
      \setlength{\tabcolsep}{5pt}
      \renewcommand{\arraystretch}{1.2}
      \caption{$5$-adic valuations of minimality defects.}
      \label{tab:5-adic-vals}
        \begin{tabular}{|c|c|c|c|c|}
          \hline
          \rowcolor{headercolor}
             $\ord_5(\twmind(E_{a,b,c}))$ & $\ord_5(e)$ & $\ord_5(A(e^2a,e^2b,ec))$ & $\ord_5(B(e^2a,e^2b,ec))$ & $\ord_5(\mind(E_{e^2a,e^2b,ec}))$ \\
             \hline
             $0$ & $0$ & $0$ & $0$ & $0$ \\
             $0$ & $1$ & $2$ & $3$ & $0$ \\
             $1$ & $0$ & $3$ \textrm{ or } $4$ & $4$ \textrm{ or } $5$ & $0$ \\
             $1$ & $1$ & $5$ \textrm{ or } $6$ & $7$ \textrm{ or } $8$ & $1$ \\
             $2$ & $0$ & $5$ & $k+7, k \geq 2$ & $1$ \\
             $2$ & $1$ & $7$ & $k+10, k \geq 2$ & $1$ \\
             $3$ & $0$ & $6$ & $9$ \text{ or } $k+8, k \geq 3$ & $1$ \\
             $3$ & $1$ & $8$ & $12$ \text{ or } $k + 11, k \geq 3$ & $2$ \\
             \hline
        \end{tabular}
    \end{table}
    \end{itemize}
\end{lemma}
\begin{proof}
    The entries of the table are obtained from using \Cref{prop:exceptional-triples} and the fact that $A(e^2a, e^2b,ec) = e^2 A(a,b,c)$ and $B(e^2a,e^2b,ec) = e^3 B(a,b,c)$.
\end{proof}
We note that the valuation of the minimality defect $\mind(E_{a,b,c})$ determines how much the upper bound on the naive height must be scaled in order to obtain the explicit leading coefficient term for the asymptotic point count estimate $N_5(H)$. To systematically understand these differences in upper bounds, we introduce some new definitions.

Pick $t \in \TMD$ and $u \colonequals (u_2, u_5) \in \{0,1\}^{\oplus 2}$. Given a minimal triple $(a,b,c) \in \GG \langle \Q \rangle$, we define
\begin{equation*}
    (a,b,c)_u \colonequals \left( \frac{a}{2^{2u_2} 5^{2u_5}}, \frac{b}{2^{2u_2} 5^{2u_5}}, \frac{c}{2^{u_2} 5^{u_5}} \right) \in \FF \langle \Q \rangle.
\end{equation*}
We denote by $g_{t,u}(n)$ the \cdef{minimality defect arithmetic height function} defined as
\begin{equation}
    g_{t,u}(n) \colonequals \#\left\{ (a,b,c) \in \GG\langle \Q \rangle : \twmind ((a,b,c)_u) = t \text{ and } |c| = n \right\}.
\end{equation}
Associated to the arithmetic height function is the \cdef{minimality defect height zeta function} 
\begin{equation}
    G_{t,u}(s) \colonequals \sum_{n=1}^\infty \frac{g_{t,u}(n)}{n^s}.
\end{equation}
We summarize the analytic properties of $G_{t,u}(s)$ as follows.
\begin{theorem}
\label{thm:md-relation}
    The following statements hold for every $t \in \twmind$ and $u \in \{0,1\}$.
    \begin{enumerate}
        \item In the half-plane $\mathrm{Re}(s) > 1$ we have
        \begin{equation*}
            G_{t,u}(s) = \begin{cases}
                4 \frac{\zeta(s)}{\zeta(2s)} \cdot \frac{1}{1+2^{-s}} \cdot \frac{1}{1+5^{-s}} \cdot F_{t}(s) &\text{ if } u = (0,0), \\
                4 \frac{\zeta(s)}{\zeta(2s)} \cdot \frac{2^{-s}}{1+2^{-s}} \cdot \frac{1}{1+5^{-s}} \cdot F_{t}(s) &\text{ if } u = (1,0), \\
                4 \frac{\zeta(s)}{\zeta(2s)} \cdot \frac{1}{1+2^{-s}} \cdot \frac{5^{-s}}{1+5^{-s}} \cdot F_{t}(s) &\text{ if } u = (0,1), \\
                4 \frac{\zeta(s)}{\zeta(2s)} \cdot \frac{2^{-s}}{1+2^{-s}} \cdot \frac{5^{-s}}{1+5^{-s}} \cdot F_{t}(s) &\text{ if } u = (1,1).
            \end{cases}
        \end{equation*}
        \item The function $G_{t,u}(s)$ admits meromorphic continuation to the half plane $\text{Re}(s) > 1/2$ with a triple pole at $s = 1$ and no other singularities.
    \end{enumerate}
\end{theorem}
\begin{proof}
    We proceed as in \Cref{sec:proof-gg-count}. For every square free $e \in \Z$, let
    \begin{equation*}
        g_{t,u}^{(e)}(n) \colonequals \#\{(a,b,c) \in \Z^3 \colon (e^2a, e^2b, ec) \in \GG \langle \Q \rangle, \twmind(E_{a,b,c}) = t, \ord_2(e) = u_2, \ord_5(e) = u_5\}.
    \end{equation*}
    Then by definition, we have
    \begin{equation*}
        g_{t,u}^{(e)}(n) = \begin{cases}
            2 f_{t}(n/|e|) &\text{ if } e \mid n, \ord_2(e) = u_2, \text{ and } \ord_5(e) = u_5, \\
            0 &\text{ otherwise}.
        \end{cases}
    \end{equation*}
    We then use the relation
    \begin{equation*}
        g_{t, u}(n) = \sum_{\substack{e \in \Z \\ \ord_2(e) = u_2 \\ \ord_5(e) = u_5 }} \mu(e)^2g_{t,u}^{(e)}(n) = 4 \sum_{\substack{e > 0 \\ \ord_2(e) = u_2 \\ \ord_5(e) = u_5 }} \mu^2(e) f_{t}(n/e)
    \end{equation*}
    to complete the proof of (1). Statement (2) follows from the identity $F(s) = \zeta(s)^2 P(s)$.
\end{proof}
We now have all the ingredients to prove \Cref{thm:main-result}.
\begin{proof}[Proof of \Cref{thm:main-result}]
    For $t = 1$ and $u = (0,0)$ we set $H_{t,u} \colonequals H$, and for general $(t,u)$ the corresponding upper bounds on naive heights $H_{t,u}$, can be computed as in \Cref{table:heights} using Lemma \ref{lemma:md-triples}. The coefficients for $H$ are determined by the minimality defects $\mind(E_{a,b,c})$ obtained for each choice of $t$ and $u$.
    \begin{table}[ht]
    \centering \setlength{\arrayrulewidth}{0.2mm} \setlength{\tabcolsep}{5pt}
    \renewcommand{\arraystretch}{1.2}
        \begin{tabular}{|c|c|c|c|c|c|c|c|c|}
        \hline
        \rowcolor{headercolor}
            $(u_2, u_5)$ & 1 & 2 & 5 & 10 & 25 & 50 & 125 & 250 \\
        \hline
            (0,0) & $H$ & $H$ & $H$ & $H$ & $5^{12}H$ & $5^{12} H$ & $5^{12} H$ & $5^{12} H$\\
            (1,0) & $H$ & $2^{12} H$ & $H$ & $2^{12} H$ & $5^{12}H$ & $10^{12} H$ & $5^{12} H$ & $10^{12} H$\\
            (0,1) & $H$ & $H$ & $5^{12} H$ & $5^{12} H$ & $5^{12}H$ & $5^{12} H$ & $25^{12} H$ & $25^{12} H$\\
            (1,1) & $H$ & $2^{12} H$ & $5^{12} H$ & $10^{12} H$ & $5^{12}H$ & $10^{12} H$ & $25^{12} H$ & $50^{12} H$\\
            \hline
        \end{tabular}
        \caption{Upper bound $H_{t,u}$ depending on $t$ and $u$.}
        \label{table:heights}
    \end{table}
    
    By  \Cref{thm:md-relation}, and \Cref{eqn:theta-psi-t}, the limits $g_{1,t,u} \colonequals \lim_{s \to 1} (s-1)^3 G_{t,u}(s)$ for each $t$ and $u$, are computed as in \Cref{table:residues}.
    \begin{table}[ht]
    \centering \setlength{\arrayrulewidth}{0.2mm} \setlength{\tabcolsep}{5pt}
    \renewcommand{\arraystretch}{1.2}
        \begin{tabular}{|c|c|c|c|}
        \hline
        \rowcolor{headercolor}
            $(u_2, u_5)$ & $t = 1,2$ & $t = 5,10$ & $t = 25,50,125,250$ \\
        \hline
            (0,0) & $\frac{5}{\pi^2} \cdot f_1$ & $\frac{8}{3\pi^2} \cdot f_1$ & $\frac{1}{3\pi^2} \cdot f_1$\\[5pt]
            (1,0) & $\frac{5}{2\pi^2} \cdot f_1$ & $\frac{4}{3\pi^2} \cdot f_1$ & $\frac{1}{6\pi^2} \cdot f_1$\\[5pt]
            (0,1) & $\frac{1}{\pi^2} \cdot f_1$ & $\frac{8}{15\pi^2} \cdot f_1$ & $\frac{1}{15\pi^2} \cdot f_1$\\[5pt]
            (1,1) & $\frac{1}{2\pi^2} \cdot f_1$ & $\frac{4}{15\pi^2} \cdot f_1$ & $\frac{1}{30\pi^2} \cdot f_1$\\[5pt]
            \hline
        \end{tabular}
        \caption{Values of $g_{1,t,u} \colonequals \lim_{s \to 1} (s-1)^3 G_{t,u}(s)$ depending on $t$ and $u$}
        \label{table:residues}
      \end{table}
    Given a squareish region $\Omega \subset \mathbb{C}$, $t \in \TMD$, and $u \in \{0,1\}^{\oplus 2}$, we define $\phi_{t,u}: \mathcal{I} \to \mathbb{C}$ to be the indicator function of those ideals $\mfa \subset \Z[i]$ giving rise to a minimal triple $(a,b,c)$ with $\twmind(E_{(a,b,c)_u}) = t$. We denote by
    \begin{equation*}
        L(\phi_{t,u}, s) := \sum_{\mfa \in \mathcal{I}} \phi_t(\mfa) (N\mfa)^{-s}
    \end{equation*}
    the corresponding Dirichlet series. By the main analytic lemma (\Cref{lemma:main-analytic-lemma}), 
    \Cref{thm:main-twerr-rigid}, and the relation $L(\phi_{t,u},s) = \frac{1}{4} G_{t,u}(4s)$, we obtain that for each $t \in \TMD$ and
    $u \in \{0,1\}^{\oplus 2}$ there exists a monic polynomial $P_{t,u}(T)$ of degree
    two such that
    \begin{equation}
        N_{\Omega_5}(H; \phi_{t,u}) = \frac{V_5}{128 \pi} \cdot g_{1,t,u} \cdot H^{1/2} P_{t,u}(\log H) + O(H^{1/2}(\log H)^{-(1-c)}), \text{ as } H \to \infty,
    \end{equation}
    where $c = \frac{3\sqrt{3}}{4\pi}$, obtained from \Cref{lemma:MAL-condition2}.

    We now use the summation
    \begin{equation}
      N_5(H) = \sum_{t \in \TMD} \sum_{u \in \{0,1\}^{\oplus 2}} N_{\Omega_5} \left({H_{t,u}^{\frac{1}{3}}}; \phi_{t,u} \right)
    \end{equation}
    to conclude that there exists a monic polynomial $P_5(T)$ of degree $2$ such that
    \begin{equation}
      N_5(H) = \frac{41}{128 \pi^3} \cdot V_5 \cdot f_1 \cdot H^{1/6} P_5(\log H) + O(H^{1/6} (\log H)^{-(1-c)}), \text{ as } H \to \infty.
    \end{equation}
\end{proof}

\subsection{The constant term}
\label{sec:constant-term}

Let $\omega_5(\theta) > 0$ be the function giving the radius of the region $\Omega_5$
defined in \Cref{eq:R5}:
\begin{equation*}
  \Omega_5(\theta) \colonequals \frac{1}{12} \frac{1}{\left( \max\{1/2 \cdot |123 \cos \theta + 114 \sin \theta + 125|^3, |2502 \cos \theta + 261 \sin \theta + 2500|^2\} \right)^{\frac{1}{3}}}.
\end{equation*}
A numerical approximation of the integral below, implemented by Steven Charlton on Mathematica, by the authors on Magma, and by Languasco and Moree on Pari/GP, is given by
\begin{equation}
  \label{eq:vol-RR5-1/4}
  V_5 := \frac{1}{2}\int_{0}^{2\pi} \sqrt{\omega_5(\theta)} \dd\theta \approx 0.09711540694426736327\dotsc.
\end{equation}


\providecommand{\bysame}{\leavevmode\hbox to3em{\hrulefill}\thinspace}
\providecommand{\MR}{\relax\ifhmode\unskip\space\fi MR }
\providecommand{\MRhref}[2]{%
  \href{http://www.ams.org/mathscinet-getitem?mr=#1}{#2}
}
\providecommand{\href}[2]{#2}

\end{document}